\documentclass[11pt]{article}

\usepackage{fullpage}

\usepackage{authblk}
\title{So Long Sucker: Endgame Analysis\thanks{Research supported in part by NSERC. Marie Rose Jerade was partly supported by OGS.}}
\author[1]{Jean-Lou De Carufel}
\author[2]{Marie Rose Jerade}
\affil[1]{School of Electrical Engineering and Computer Science, University of Ottawa}
\affil[2]{Department of Mathematics and Statistics, University of Ottawa}

\usepackage[T1]{fontenc}
\usepackage{tikz-cd}
\usepackage{blindtext}

\usepackage{hyperref,xcolor}
\usepackage{amsmath}
\usepackage{subcaption}
\usepackage{graphicx}
\usepackage{amsfonts}
\usepackage{amssymb}
\usepackage{amsthm}
\allowdisplaybreaks[4]
\usepackage{color}
\usepackage{mathtools}
\usepackage{enumitem}
\usepackage{epstopdf}
\usepackage{polynom}
\usepackage{babel}
\usepackage{mathrsfs}
\usepackage{chngcntr}
\usepackage{verbatim} 
\usepackage[]{mdframed}
\usepackage{listings}
\usepackage{pdflscape}
\usepackage{float}
\usepackage{rotating}
\usepackage{pdfpages}
\usepackage{xcolor}
\usepackage{diagbox}
\usepackage{soul}
\usepackage{multirow} 
\usepackage{fancybox,framed}

\hypersetup{
	colorlinks=true,
	linkcolor=blue,
	citecolor=magenta,
	urlcolor=cyan}

\newtheorem{theorem}{Theorem}[section]
\newtheorem{definition}[theorem]{Definition}
\newtheorem{prop}[theorem]{Proposition}
\newtheorem{cor}[theorem]{Corollary}
\newtheorem{lemma}[theorem]{Lemma}

\newcommand{\strat}{\mathcal{S}}

\begin{document}





\maketitle

\begin{abstract}
\hspace{\parindent} \emph{So Long Sucker} is a strategy board game that requires 4 players, each with $c$ chips of their designated color, and a board made of $k$ empty piles. With a clear set-up comes intricate rules, such as: players taking turns but not in a fixed order, agreements made between some players broken at any time, or a player winning the game without any chips in hand.

One of the main points of interest in studying this game is finding when a player has a winning strategy. The game begins with four players who get successively eliminated until only the winner is left. To study winning strategies, it is of interest to look at endgame situations. For that, we study the following game set-up: there are two players left in the game, Blue and Red, with only their respective chip colors. In this paper, we characterize Blue's winning scenarios and strategies for this game set-up through a delicate case analysis.
\end{abstract}

\pagenumbering{arabic}

\section{So Long Sucker: An Introduction}
\label{section.introduction}

In 1964, the game \emph{So Long Sucker} was developed by Mel Hausner, John Nash, Lloyd Shapley, and Martin Shubik \cite{SLS64}. It is a board game that requires $4$ players, each with the same number of chips of their designated color, and a board made of empty piles. So Long Sucker (\emph{SLS} henceforth) is a \emph{perfect information} board game \cite{Ross24}, that is, all players know the same information at any moment in the game. It is a \emph{deterministic} game \cite{Hodges24}, that is, a game without an instance of randomness. These two properties classify SLS as a \emph{combinatorial game} \cite{Siegel13}. Furthermore, the game is \emph{non-constant} \cite{Shubik02b}, that is, a game in which the sum of the gains and losses of the players is not fixed. Although the game is not widely known, the simple layout and detailed rules of the game make the game an ideal tool for depicting basic problems in conflict, coordination, and cooperation \cite{Shubik87}. These characteristics may remind the reader of common board games, like Risk or Diplomacy, where the players are encouraged to form coalitions and eventually break them in order to win. SLS illustrates one of the most profound game theory problems: ``the clash between individual, non-social, rational optimizing behavior, and personality and passion'' \cite{Shubik02b}. 

The game has peaked the interest of many researchers. They talk about the peculiarity of the game or draw inspiration from it. However, to the best of our knowledge, none of them have attempted to solve the game. SLS has been described ``to help form impressions
of what it feels like to be in a dog-eat-dog world'' \cite{Rapoport1971},
as ``vicious'' \cite{Poundstone1992}, ``anti-chess'' \cite{Burnett2012}, and ``fiendish'' \cite{Tannenbaum2014}. Moreover, the game inspired the first episode of filmmaker Adam Curtis's \emph{The Trap} documentary \cite{Curtis07}. The episode explores the role that game theory played during the Cold War and how mathematical models of human behavior affected economic and political decisions. Subsequently, Hofstede and Tipton studied the individual and group behavior of players while playing SLS~\cite{HofstedeTiptonMurff11}. Their observations led them to reason that the game dynamics depend on more than just the defined rules. They argue that the personalities and cultures of the participants are of crucial importance in how players act throughout the game. Moreover, Guerra-Pujol makes a parallel comparison between characters from the TV show Breaking Bad and the rules of SLS \cite{Guerra-Pujol17}. More recently, Adak and Sharan \cite{AdakSharan2024} applied classical deep reinforcement learning (DRL) to SLS. The study highlights DRL's limitations in complex social games and suggests future improvements using advanced or hybrid methods. 

As in many other games, it is of interest to establish winning strategies for the players. The rules defined in the next section show that there are many possible deviations throughout a game of SLS. This, combined with the huge number of game states, makes the game difficult to analyze. As mentioned earlier, there are a variety of papers that discuss the usefulness and potential of SLS in studying human behavior, and beyond. However, no paper studies this game from a mathematical point of view. In this paper, we are interested in studying endgame situations. In particular, we will study situations where there are two players left in the game with only their respective chip colors. We establish the necessary and sufficient conditions with which the remaining two players can guarantee to win. 

In Sections~\ref{section.setup} and~\ref{subsec:RulesGame}, we describe the rules of the game. We provide a short example of how the game plays out in Section~\ref{section.example}. Then we explain how the rules apply when only two players and two colors remain in Section~\ref{section.endgame.2.players}. Finally, in the rest of the paper, we proceed with the analysis of the endgame situation with two players and two colors.

\subsection{Setup}
\label{section.setup}

A game of SLS consists of four players, each designated by a different color, namely Blue, Red, Green and Yellow. At the beginning of the game, each player has exactly $c$ \emph{chips} of their designated color, where $c$ is a positive integer.
The game is played on a board consisting of $k$ ordered stacks of chips, where $k$ is a positive integer\footnote{In Hausner et al.'s original version, $c = 7$ and $k$ is not specified~\cite{SLS64}. Observe that in general, since there are $4c$ chips in total, fixing any $k \geq 4c$ corresponds to having an unlimited number of piles}. These stacks are referred to as \emph{piles}.
At the beginning of the game, all piles are empty, as seen in Figure~\ref{fig:EmptyBoard}.
\begin{figure}[H]
\centering
\includegraphics[scale=1]{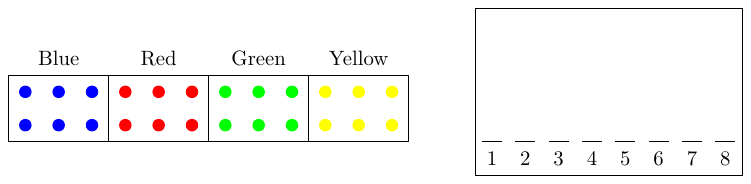}
\caption{SLS setup with $c=6$ and $k=8$. The symbol $\_$ illustrates the bottom of a pile.}
\label{fig:EmptyBoard}
\end{figure}
A cup, referred to as the \emph{dead box}, is placed near the board. 
When we say that a player \emph{discards} a chip, we mean that a player places the chip in the deadbox.
The chips in the dead box are out of the game and are completely removed from play.
Let $\mathscr{X}$ be a player with designated color $x$.
Each chip in $\mathscr{X}$'s possession is called a \emph{guard} if it is of color $x$, and a \emph{prisoner} if it is of a color different from $x$.

During the game, the player designated as the \emph{active player} must place one chip in their possession on top of a pile, referred to as the \emph{played pile}. If the active player cannot play due to a lack of chips, then they are eliminated.
The placement of the chip may result in a capture of the played pile by one of the players, as we will see in the next subsection.
Afterwards, the next active player is chosen by the current active player.

Moreover, at any moment during the game, any player (including the active player) may \emph{donate} a prisoner to another player, \emph{discard} a prisoner or \emph{make a deal} with any other players. Each of these moves is governed by a specific rule, as we will see in the next subsection. Note that deals are non-enforceable (refer to the Deal Rule), which classifies the game as \emph{non-cooperative} \cite{VanDamme2001}. The winner is the last player remaining in the game, i.e., that has not been eliminated.

\subsection{The Rules of the Game}
\label{subsec:RulesGame}

Here is the exact description of SLS. We first introduce the rules of the game as they were originally described, with four players and their respective colors. Aside from the Next Player Rule (see below), the rules are fairly straightforward. In the next section, we will see that the rules simplify once we restrict ourselves to two players and their respective colors.

\paragraph{First Player:} At the beginning of the game, the active player is chosen randomly or agreed upon among all players.

\paragraph{Active Player's Turn:}
As soon as a player is designated as the active player, their \emph{turn} starts.
Let $\mathscr{X}$ be the active player.
During their turn,
$\mathscr{X}$ has to (1) place one chip on the board and then (2) decide who is going to be the next active player.
\begin{enumerate}
\item {\bf Placing One Chip on the Board}:
\begin{enumerate}
\item If $\mathscr{X}$ starts their turn with at least one chip in their possession, then $\mathscr{X}$ must place one chip on the board, on top of a pile $\pi$.
If, after $\mathscr{X}$ places their chip on the board, $\pi$ has at least two chips and the two chips at the top of $\pi$ have the same color, say $y$,
then the player with designated color $y$ must comply with the Capture Rule (see below).

\item Otherwise, if $\mathscr{X}$ starts their turn without any chips in their possession, then they face elimination.
If another player is willing to apply the Donation Rule (see below) in order to donate a chip $c$ to $\mathscr{X}$, then $\mathscr{X}$ can place $c$ on the board and complete their turn.
Otherwise, when every player has declared their refusal to donate a chip to $\mathscr{X}$, then $\mathscr{X}$ is eliminated.
The next active player must then be chosen according to the Next Player Rule (see below).
\end{enumerate}

\item {\bf Designate the Next Active Player}:
After $\mathscr{X}$ places a chip on the board, they must designate the next active player. For that, $\mathscr{X}$ must comply with the Next Player Rule (see below). As soon as this is done,
$\mathscr{X}$'s turn ends and the next active player starts their turn.
\end{enumerate}

\paragraph{\emph{Capture Rule}:}
Let $\pi$ be a non-empty pile whose chip on top has color $y$. If the active player places a chip of color $y$ on top of $\pi$, then $\pi$ now contains one more chip and the following Capture Rule applies.
\begin{itemize}
\item  If player $\mathscr{Y}$ with designated color $y$ is eliminated, then all chips in $\pi$ must be discarded.

\item Otherwise, $\mathscr{Y}$ must \emph{capture} $\pi$.
That is, $\mathscr{Y}$ must discard exactly one chip from $\pi$ and take possession of all the remaining chips in $\pi$.
\end{itemize}

\paragraph{\emph{Next Player Rule}:}
Let $\mathscr{X}$ be the active player.

If $\mathscr{X}$ starts their turn without any chips in their possession with no other player willing to apply the Donation Rule (see below), then $\mathscr{X}$ is eliminated. The next active player must be the player $\mathscr{Y}$ who gave $\mathscr{X}$ their turn. If $\mathscr{Y}$ is already eliminated, then the \emph{Special Rule} must be applied, see below.

Otherwise, after placing a chip on top of a pile, $\mathscr{X}$ must designate the next active player based on the following conditions.
\begin{enumerate}
\item If the chip that $\mathscr{X}$ places on the board results in a capture, then the player $\mathscr{Y}$ designated by the color making the capture becomes the next active player. If $\mathscr{Y}$ is already eliminated, then the \emph{Special Rule} must be applied, see below.
\item If the chip that $\mathscr{X}$ places on the board does not lead to a capture, then $\mathscr{X}$ chooses the next active player based on the following conditions. Let $\mathcal{P}$ be the set of non-eliminated players and $\mathcal{C}$ be the set of their designated colors.
\begin{enumerate}
\item If there is at least one chip of each color in $\mathcal{C}$ in the played pile, then the move automatically goes to the player in $\mathcal{P}$ whose highest chip in the played pile is the lowest among the other chips belonging to players in $\mathcal{P}$.
\item If not all colors from $\mathcal{C}$ are represented in the played pile, then $\mathscr{X}$ can choose any player from $\mathcal{P}$ whose designated color is not represented in the played pile.
\end{enumerate}

\item \emph{Special Rule}: If an eliminated player $\mathscr{Y}$ is designated as the next active player, then the player $\mathscr{Z}$ that gave $\mathscr{Y}$ their last move becomes the active player\footnote{If $\mathscr{Z}$ is already eliminated, then the Special Rule must be applied to $\mathscr{Z}$.}.
\end{enumerate}

\paragraph{\emph{Donation Rule}:}
Let $\mathscr{Y}$ be a player. At any moment during the game, $\mathscr{Y}$ may donate a prisoner to any player that is not eliminated.
Moreover, if the active player $\mathscr{X}\neq\mathscr{Y}$ starts their turn without any chips in their possession, then $\mathscr{Y}$ is allowed to donate a prisoner to $\mathscr{X}$ to prevent $\mathscr{X}$'s elimination. 

\paragraph{\emph{Discarding Rule}:}
Let $\mathscr{Y}$ be a player. At any moment during the game, $\mathscr{Y}$ may discard a prisoner.

\paragraph{\emph{Deal Rule}:}
At any moment during the game, any number of non-eliminated players may make non-enforceable deals together, as long as the rules of the game are respected. The deal must be made and agreed upon in front of all players, remain within the current game, and be directly related to the game in play. Moreover, the deal cannot involve an eliminated player.
\emph{Here is an example of a deal: Yellow and Green can decide to work together in order to eliminate Red. At any moment, Green or Yellow can decide to break that deal.}

\paragraph{The Winner:}
The winner is the last player
remaining in the game.

\paragraph{Remarks}
All of the following statements are direct consequences of the rules we presented.
\begin{itemize}
\item It is possible for a player to possess chips that are not of their designated color. Recall that these chips are called prisoners.
\item It is possible for a player $\mathscr{X}$ to discard a guard. The only way for this to happen is for $\mathscr{X}$ to capture a pile containing a chip of color $x$.
By the Capture Rule, $\mathscr{X}$ can then discard one guard.
\item It is not possible for a player to donate a guard.
\item It is possible for a player to win the game without having
any chips left in their possession.
\item It is possible for a player to be the active player more than once in a row, as long as the active player complies with the Next Player Rule.
\item During an active player's turn, at most one capture can occur.
\end{itemize}

\subsubsection{Example}
\label{section.example}

Let us consider the following example where after a few turns, Yellow gets designated as the active player (refer to Figure~\ref{fig:example11}).
\begin{figure}[H]
\centering
\includegraphics[page=1,scale=1]{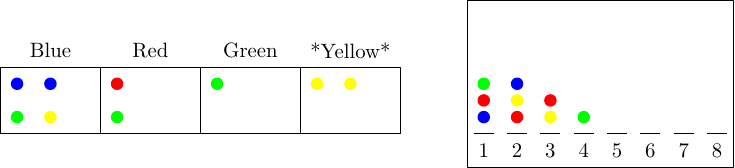}
\caption{Yellow is the active player.}
\label{fig:example11}
\end{figure}
Yellow places a yellow chip on pile $1$ after which Blue automatically becomes the active player (refer to the Next Player Rule and Figure~\ref{fig:example12}).
\begin{figure}[H]
\centering
\includegraphics[page=2,scale=1]{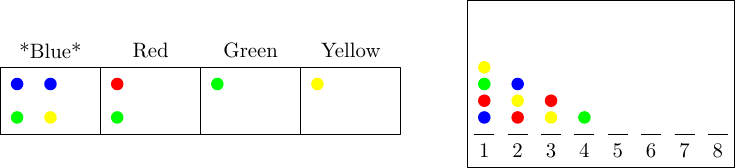}
\caption{Blue is the active player.}
\label{fig:example12}
\end{figure}
Blue places a yellow chip on pile $4$. At the same time, Red donates a green chip to Yellow (refer to the Donation Rule). Since only a green chip and a yellow chip are on the played pile, then Blue can give the turn back to themselves or to Red (refer to the Next Player Rule). They decide to give the turn to Red, who becomes the active player (refer to Figure~\ref{fig:example13}).
\begin{figure}[H]
\centering
\includegraphics[page=3,scale=1]{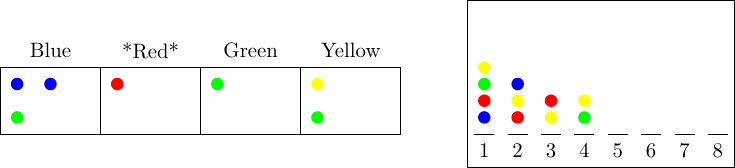}
\caption{Red is the active player.}
\label{fig:example13}
\end{figure}
Red places a red chip on pile $3$. As such, Red must comply with the Capture Rule. They decide to discard a red chip and take possession of the one red and one yellow chips remaining in the pile. By the Next Player Rule, they get the next turn
(refer to Figure~\ref{fig:example14}). 
\begin{figure}[H]
\centering
\includegraphics[page=4,scale=1]{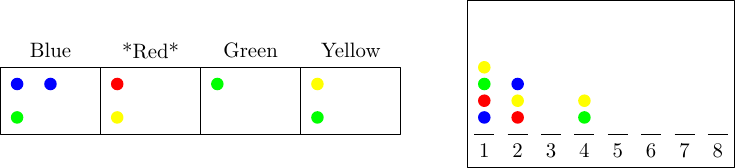}
\caption{Red is the active player.}
\label{fig:example14}
\end{figure}

\section{Endgame Situations with Two Players and Two Colors}
\label{section.endgame.2.players}

For the rest of this paper, we assume without loss of generality that the two players remaining in the game are Blue and Red. Moreover, we assume that every chip in the player's possession as well as on the board is blue or red.
We refer to such situations by saying that \emph{two players and two colors are left}.
Observe that each remaining player might have prisoners in their possession, and they may not have the same number of guards in their possession.
For the rest of this paper, we denote the board by $\mathbb{B}$, as well as the blue and red players by $\mathscr{B}$ and $\mathscr{R}$, respectively. Moreover, we denote the blue and red chips by $b$ and $r$, respectively.
Let $\mathscr{X} \in\{\mathscr{B}, \mathscr{R}\}$ be a player. When we say that \emph{$\mathscr{X}$ is the first player}, we mean that, at the moment that both Green and Yellow have been eliminated, and all green and yellow chips have been discarded, $\mathscr{X}$ is designated as the active player.

When two players and two colors are left, the rules of SLS are simpler to apply. For instance, we will show that, whatever move the active player makes, the Next Player Rule reduces to only one option for the next active player (refer to Theorems~\ref{thm:SameActive} and~\ref{thm:DifferentActive}).

For the rest of this paper, we refer to a pile with an $x$ chip on top as an \emph{$x$-pile}.
Moreover, we refer to piles of length at least $2$ as \emph{long piles}, and piles with exactly one chip as \emph{singletons}. A long pile with an $x$ chip on top is referred to as a \emph{long $x$-pile}, and a singleton whose unique chip is of color $x$ is referred to as an \emph{$x$-singleton}. 

\begin{theorem}[Same Active Player]
\label{thm:SameActive}
Let $\mathscr{X} \in\{\mathscr{B}, \mathscr{R}\}$ be a player with designated color $x$ and let $y\in\{b,r\}$ be a color such that $y\neq x$. If $\mathscr{X}$ is the active player and makes any of the following three moves (assuming they have the appropriate chips in their possession), then they get designated as the next active player.
\begin{enumerate}
\item $\mathscr{X}$ places a $y$ chip on an empty pile,
\item $\mathscr{X}$ places a $y$ chip on an $x$-pile,
\item or $\mathscr{X}$ places an $x$ chip on an $x$-pile.
\end{enumerate}
\end{theorem}

\begin{proof}
Suppose $\mathscr{X}$ is the active player.
\begin{enumerate}
\item If $\mathscr{X}$ places a $y$ chip on an empty pile,
then $\mathscr{X}$ is the only player remaining in the game whose chip color is not in the played pile. Thus, $\mathscr{X}$ gets the next turn by the the Next Player Rule.

\item If $\mathscr{X}$ places a $y$ chip on an $x$-pile, then $\mathscr{X}$ is the player whose chip color is furthest down in the pile. Therefore, $\mathscr{X}$ gets the next turn by the Next Player Rule.

\item If $\mathscr{X}$ places an $x$ chip on an $x$-pile,
then $\mathscr{X}$ must comply with the Capture Rule and then gets the next turn by the Next Player Rule.
\end{enumerate}
\end{proof}

\begin{theorem}[Different Active Player]
\label{thm:DifferentActive}
Let $\mathscr{X},\mathscr{Y} \in\{\mathscr{B}, \mathscr{R}\}$ be two different players with designated colors $x$ and $y$, respectively. If $\mathscr{X}$ is the active player and makes any of the following three moves (assuming they have the appropriate chips in their possession), then $\mathscr{Y}$ gets designated as the next active player.
\begin{enumerate}
\item $\mathscr{X}$ places an $x$ chip on an empty pile,
\item $\mathscr{X}$ places an $x$ chip on a $y$-pile,
\item or $\mathscr{X}$ places a $y$ chip on a $y$-pile.
\end{enumerate}
\end{theorem}

\begin{proof}
Suppose $\mathscr{X}$ is the active player.
\begin{enumerate}
\item If $\mathscr{X}$ places an $x$ chip on an empty pile, then $\mathscr{Y}$ is the only player remaining in the game whose chip color is not in the played pile. Thus, $\mathscr{Y}$ gets the next turn by the the Next Player Rule.

\item If $\mathscr{X}$ places an $x$ chip on a $y$-pile,
then $\mathscr{Y}$ is the player whose chip color is furthest down in the pile. Therefore, $\mathscr{Y}$ gets the next turn by the Next Player Rule.

\item If $\mathscr{X}$ places a $y$ chip on a $y$-pile, then $\mathscr{Y}$ must comply with the Capture Rule and then gets the next turn by the Next Player Rule.
\end{enumerate}
\end{proof}

The following proposition implies that, when only two players and two colors are left, the chips in every pile with at least two chips have alternating colors.
\begin{prop}
\label{prop:AlternatingPile}
Let $\pi$ be a pile containing at least two chips. At the beginning of each turn, any two consecutive chips in $\pi$ have different colors.
\end{prop}

\begin{proof}
By the Capture Rule, as soon as a pile has two consecutive chips of the same color, this results in a capture. Therefore, at the beginning of each turn,
any two consecutive chips in $\pi$ have different colors.
\end{proof}

\begin{prop}
\label{prop:XcapturesXpile}
Let $\mathscr{X} \in \{\mathscr{B},\mathscr{R}\}$ be a player with designated color $x$.
Assume $\mathscr{X}$ is the active player and places an $x$ chip on top of an $x$-pile. Then the number of $x$ chips in $\mathscr{X}$'s possession would not have decreased by the end of their turn.
\end{prop}

\begin{proof}
Suppose $\mathscr{X}$ is the active player and places an $x$ chip on an $x$-pile $\pi$. Then $\mathscr{X}$ must comply with the Capture Rule.

If $\pi$ is an $x$-singleton, then $\mathscr{X}$ discards one $x$ chip from $\pi$ and keeps the other $x$ chip in their possession.
If no other player donates an $x$ prisoner to $\mathscr{X}$ (with respect to the Donation Rule), then the number of $x$ chips in $\mathscr{X}$'s possession would not have changed by the end of their turn.
Otherwise, the number of $x$ chips in $\mathscr{X}$'s possession would have increased by the end of their turn.

If $\pi$ is not an $x$-singleton, then it contains at least two chips. By Proposition~\ref{prop:AlternatingPile},
$\pi$ contains at least one $y$ chip, where $y\neq x$.
As such, $\mathscr{X}$ discards (1) an $x$ chip or (2) a $y$ chip from $\pi$.
\begin{enumerate}
\item Consider the case where $\mathscr{X}$ discards an $x$ chip.
If no other player donated an $x$ prisoner to $\mathscr{X}$ (with respect to the Donation Rule), then the number of $x$ chips in $\mathscr{X}$'s possession would not have changed by the end of their turn.
Otherwise, the number of $x$ chips in $\mathscr{X}$'s possession would have increased by the end of their turn.
\item Consider the case where $\mathscr{X}$ discards a $y$ chip. Then $\mathscr{X}$ has at least one extra $x$ chip in their possession at the end of their turn.
\end{enumerate}
\end{proof}

As we observed in Subsection~\ref{subsec:RulesGame}, a player can be the active player more than once in a row.
We define a \emph{round} as a sequence of turns during which the active player does not change.

\begin{definition}[Active Player's Round]
Let $\mathscr{X},\mathscr{Y} \in\{\mathscr{B}, \mathscr{R}\}$ be two different players.
When $\mathscr{Y}$ gives the turn to $\mathscr{X}$, we say that $\mathscr{X}$'s \emph{round} \emph{starts}.
Then, when $\mathscr{X}$ gives the turn back to $\mathscr{Y}$, we say that $\mathscr{X}$'s round \emph{ends}.
\end{definition}
When $\mathscr{X}$'s round starts, $\mathscr{X}$ first makes a (possibly empty) sequence of moves from the ones listed in Theorem~\ref{thm:SameActive}.
Then $\mathscr{X}$ ends their round either by being eliminated or by making one move from the ones listed in Theorem~\ref{thm:DifferentActive}.
Observe that during $\mathscr{X}$'s round, both players can donate or discard prisoners, and they can make a deal (refer to the Donation, the Discarding and the Deal Rules).

We define the \emph{length} of a pile $\pi$, denoted by $|\pi|$, as the number of chips contained in $\pi$. Similarly, we denote by $|\pi|_x$ the number of $x$ chips in $\pi$, where $x\in \{b,r\}$.
Let $\lambda \geq 1$ be an integer and let $x_1,x_2,...,x_{\lambda} \in \{b,r\}$. Consider an $x_{\lambda}$-pile $\pi$ of length $\lambda$, such that the color sequence of chips from bottom to top is $(x_1,x_2,...,x_{\lambda})$.
We sometimes refer to the pile $\pi$ as $(x_1,x_2,...,x_{\lambda})$ or as an $(x_1, x_2, ..., x_{\lambda})$-pile. If $\pi$ and $\pi'$ are two sequences, then $\langle \pi,\pi'\rangle$ corresponds to the concatenation of $\pi$ and $\pi'$. As such, if an $x$ chip is placed on top of $\pi$, we denote the pile we obtain by $\langle \pi, (x)\rangle$. 

\begin{definition}[Strategy $\strat$]
Let $\mathscr{X} \in\{\mathscr{B}, \mathscr{R}\}$ be a player with designated color $x$ and let $y\in\{b,r\}$ be a color such that $y\neq x$.
Assume $\mathscr{X}$ is the active player and has at least one $x$ chip in their possession.
Let $\strat$ be the strategy defined by the following sequence of moves.
\begin{enumerate}
\item Player $\mathscr{X}$ captures all $x$-piles. For each captured pile $\pi$, when applying the Capture Rule, if there is a $y$ chip in $\pi$, then $\mathscr{X}$ discards a $y$ chip. Otherwise, $\mathscr{X}$ discards an $x$ chip.
\item Player $\mathscr{X}$ discards all of their prisoners, i.e., all the $y$ chips in their possession.
\item If there is at least one $y$-pile on the board, then $\mathscr{X}$ places an $x$ chip on the $y$-pile of greatest length. Otherwise, $\mathscr{X}$ places an $x$ chip on an empty pile.
\end{enumerate}
\end{definition}
Observe that when a player applies $\strat$, they do not make deals with their opponent and they do not donate any chips to their opponent.
The following lemma implies that $\strat$ is well-defined.
\begin{lemma}
Let $\mathscr{X},\mathscr{Y} \in\{\mathscr{B}, \mathscr{R}\}$ be two different players with designated color $x$ and $y$, respectively.
Assume $\mathscr{X}$ is the active player and has at least one $x$ chip in their possession.
Then $\mathscr{X}$ can apply $\strat$,
after which $\mathscr{Y}$ becomes the active player.
\end{lemma}

\begin{proof}
Suppose $\mathscr{X}$ is the active player and has at least one $x$ chip in their possession.
\begin{enumerate}
\item If there is an $x$-pile $\pi$ on the board,
then $\mathscr{X}$ can capture it (refer to the Capture Rule).
After capturing $\pi$, $\mathscr{X}$ has at least one $x$ chip in their possession by Proposition~\ref{prop:XcapturesXpile}.
Moreover, by Theorem~\ref{thm:SameActive}, $\mathscr{X}$ is again the active player.
Therefore, they can keep capturing $x$-piles as long as there are
some on the board. Afterwards, $\mathscr{X}$ has at least one $x$ chip in their possession and $\mathscr{X}$ is again the active player.
\item Player $\mathscr{X}$ is allowed to discard all the $y$ chips in their possession (refer to the Discarding Rule).
\item When there are no more $x$-piles on the board,
$\mathscr{X}$ can proceed with Step 3 of $\strat$. Afterwards, by Theorem~\ref{thm:DifferentActive}, $\mathscr{Y}$ becomes the active player.
\end{enumerate}
\end{proof}

For the rest of this paper, the \emph{state of the board} refers to the number of empty piles, singletons and long piles there are on the board at a given moment of the game. More specifically, we denote by $\mathbb{B}=(k_e,k_r,k_b,\ell,h)$ the state where there are $k_e$ empty piles, $k_r$ $r$-singletons, $k_b$ $b$ singletons, $\ell$ long $r$-piles, and $h$ long $b$-piles on the board. The \emph{state of a player} refers to the number of $b$ chips and $r$ chips in their possession at a given moment of the game. We denote by $\mathscr{B}=(m_b,m_r)$ the state where $\mathscr{B}$ has $m_b$ $b$ chips and $m_r$ $r$ chips in their possession. Similarly, we denote by $\mathscr{R}=(n_b,n_r)$ the state where $\mathscr{R}$ has $n_b$ $b$ chips and $n_r$ $r$ chips in their possession. Finally, the \emph{state of the game (game state)} refers to the combination of the states of $\mathbb{B}$, $\mathscr{B}$ and $\mathscr{R}$ at a given moment of the game.

In this paper, we characterize the game states where there exists a winning strategy for the first player.
Moreover, we show that when there exists a winning strategy for a player $\mathscr{X}$, then $\strat$ is a winning strategy for $\mathscr{X}$.
In Section~\ref{sec:GTI}, we consider the game states where the board contains no long $b$-piles. We refer to these states as \emph{type I} and \emph{generalized type I} game states.
In Section~\ref{sec:GTII}, we consider the game states where the board contains at least one long $b$-pile and exactly one long $r$-pile. We refer to these states as \emph{type II} and \emph{generalized type II} game states.
In Section~\ref{sec:FinalTheorem}, we consider the general case.
All these cases are summarized in Table~\ref{table:summary.cases}.

\begin{table}[H]
\centering
\begin{tabular}{|l||c|c|c|c|c|}
\hline
\diagbox{type}{number of} & $r$-singletons & $b$-singletons & long $r$-piles & long $b$-piles & Section \\ \hline\hline
    type I & $\geq 0$ & $\geq 0$ & $\leq 1$ & $0$ & \ref{subsec:typeI}\\\hline
generalized type I & $\geq 0$ & $\geq 0$ & $\geq 0$ & $0$ & \ref{subsec:gentypeI} \\\hline
type II & $\geq 0$ & $\geq 0$ & $1$ & $1$ & \ref{subsec:typeII}\\\hline
generalized type II & $\geq 0$ & $\geq 0$ & $1$ & $\geq 1$ & \ref{subsec:gentypeII} \\\hline
general & $\geq 0$ & $\geq 0$ & $\geq 0$ & $\geq 0$ & \ref{sec:FinalTheorem} \\\hline
\end{tabular}
\caption{Breakdown of the endgame analysis with two players and two colors.\label{table:summary.cases}}
\end{table}

The goal of this paper is to prove the following theorem about the two players and two colors game state. 

\begin{framed}
Let $\mathbb{B}=(k_e,k_r,k_b,\ell,h)$. If $\mathbb{B}$ contains $h\geq 1$ long $b$-piles, we denote the $h$ long $b$-piles by $\beta_{1}, \beta_{2}, \ldots, \beta_{h}$. If $\mathbb{B}$ contains $\ell\geq 1$ long $r$-piles, we denote the $\ell$ long $r$-piles by $\rho_{1}, \rho_{2}, \ldots, \rho_{\ell}$. Consider the players $\mathscr{B}=\left(m_b, m_r\right)$, $\mathscr{R}=\left(n_b, n_r\right)$, and assume that $\mathscr{B}$ is the active player. Then, $\mathscr{B}$ has a winning strategy if and only if 
$$m_b>0 \qquad \textrm{and} \qquad \left( n_r=0 \quad\textrm{ or }\quad m_b+\sum_{i=1}^{h}\left|\beta_{i}\right|_b>n_r+\sum_{i=1}^{\ell}\left|\rho_{i}\right|_r-\max_{1\leq i\leq \ell}\big\{\left|\rho_{i}\right|_r\big\}\right)$$
when $\mathscr{B}$'s turn starts. Whenever there is a winning strategy for $\mathscr{B}$ (respectively for $\mathscr{R}$), then $\strat$ is such a strategy.
\end{framed}

\section{Boards Without Long $b$-Piles}
\label{sec:GTI}

In this section, we analyze game states with two players and two colors when the board does not contain any long $b$-piles.

\subsection{Type I Board}
\label{subsec:typeI}

We start by studying game states where the board does not contain any long $b$-piles and contains at most one long $r$-pile.
\begin{definition}
\label{def:typeI}
The board $\mathbb{B}$ is said to be of \emph{type I} if $\mathbb{B}=(k_e,k_r,k_b,\ell,0)$ with $\ell\in\{0,1\}$ (refer to Figure~\ref{fig:type_i}).
\begin{figure}[H]
     \centering
     \begin{subfigure}[h]{0.3\textwidth}
         \centering
         \includegraphics[scale=1]{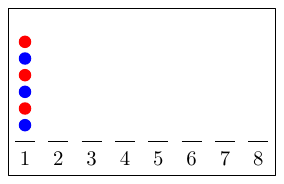}
         \label{fig:type_i_1}
     \end{subfigure}
     \hfill
     \begin{subfigure}[h]{0.3\textwidth}
         \centering
         \includegraphics[scale=1]{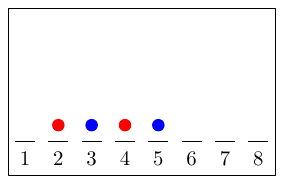}
         \label{fig:type_i_2}
     \end{subfigure}
     \hfill
     \begin{subfigure}[h]{0.3\textwidth}
         \centering
         \includegraphics[scale=1]{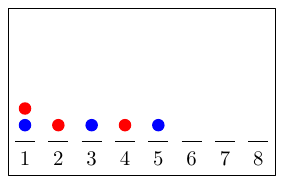}
         \label{fig:type_i_3}
     \end{subfigure}
        \caption{Examples of type I boards with $8$ piles.}
        \label{fig:type_i}
\end{figure}
\end{definition}

In all the proofs in this paper, because of the Donation and Discarding Rules,
we need to be careful when stating the number of chips in a player's possession.
For instance,
suppose $\mathscr{B} = (m_b,0)$ (with $m_b \geq 1$).
After playing a $b$ chip, if we do not know whether $\mathscr{R}$ applied the Donation Rule during $\mathscr{B}$'s round, then we can only say that $\mathscr{B}$ has \emph{at least} $m_b-1$ $b$ chips.
In such a case, we write $\mathscr{B} = (m_b',0)$, where $m_b' \geq m_b-1$.

In this subsection, we first describe the states where there is a winning strategy for $\mathscr{B}$ on a type I board (refer to Proposition~\ref{prop-typeI-B}). Then we describe the states where there is a winning strategy for $\mathscr{R}$ on a type I board (refer to Proposition~\ref{prop-typeI-R}). We finally present the characterization of the winning states for the first player in Theorem~\ref{thm-type1}.

\begin{prop}
\label{prop-typeI-B}
Let $\mathbb{B}=(k_e,k_r,k_b,\ell,0)$ be of type I, $\mathscr{B}=\left(m_b, m_r\right)$,  $\mathscr{R}=\left(n_b, n_r\right)$, and assume that $\mathscr{B}$ is the active player. If $m_b>n_r \geq 0$ when $\mathscr{B}$'s turn starts, then $\mathscr{B}$ wins by applying $\strat$ at every round.
\end{prop}

In the following proof and for the rest of this paper, we assume that an empty sum is equal to $0$, for instance $\sum_{i = 2}^1 i = 0$.
We prove Proposition~\ref{prop-typeI-B} using the following inductive argument on $n_b+n_r$. 
\begin{description}
\item[Base case:] We prove that for all game states where $\mathbb{B}$ is of type I, $\mathscr{B}=\left(m_b, m_r\right)$ with $m_b > 0$, and $\mathscr{R}=\left(0, 0\right)$,
if $\mathscr{B}$ is the active player, then $\mathscr{B}$ wins by applying $\strat$.
\item[Inductive step:] Consider a game state where $\mathbb{B}=(k_e,k_r,k_b,\ell,0)$ is of type I, $\mathscr{B}=\left(m_b, m_r\right)$, and $\mathscr{R}=\left(n_b, n_r\right)$, such that $m_b > n_r \geq 0$ and $n_b+n_r > 0$.
Assume $\mathscr{B}$ is the active player and they apply $\strat$.
Then, no matter what $\mathscr{R}$ plays during their round, when $\mathscr{R}$'s round is over, we get $\mathbb{B}=(k_e'',k_r'',k_b'',\ell'',0)$, $\mathscr{B}=\left(m_b'', m_r''\right)$ and  $\mathscr{R}=\left(n_b'', n_r''\right)$, where $\ell''\in\{0,1\}$, $\mathscr{B}$ is the active player, $m_b'' > n_r'' \geq 0$, and $n_b''+n_r'' < n_b+n_r$.
\end{description}

\begin{proof}
Suppose $\mathbb{B}=(k_e,k_r,k_b,\ell,0)$ is of type I.

Consider a game state where $n_b+n_r = 0$ (base case). Therefore $n_b=n_r=0$, which means $\mathscr{R}$ has no chips and $\mathscr{B}$ has at least one $b$ chip since $m_b>n_r$. Player $\mathscr{B}$ applies Strategy $\strat$, after which the move goes to $\mathscr{R}$ who has no chips. Thus, $\mathscr{R}$ is eliminated and $\mathscr{B}$ wins.

Consider now a game state where $\mathscr{B}=\left(m_b, m_r\right)$
and $\mathscr{R}=\left(n_b, n_r\right)$ with $m_b > n_r \geq 0$ and $n_b+n_r > 0$.
Since $m_b>n_r \geq 0$, then $\mathscr{B}$ has at least one $b$ chip. Player $\mathscr{B}$ applies Strategy $\strat$: 
they capture $k_b$ $b$-singletons and
discard all their prisoners as per steps 1 and 2 of $\strat$.
Then, as per Step 3 of $\strat$, they place a $b$ chip on the only long $r$-pile if $\ell =1$, on an $r$-singleton if $\ell =0$ and $k_r>0$, or on an empty pile otherwise. In all three cases, we have $\mathscr{B}=\left(m_b', 0\right)$, where $m_b' \geq m_b - 1$, and $\mathscr{R}$'s round starts with the state of the board $\mathbb{B}$ being the following, depending on the values of $\ell$ and $k_r$.
\begin{center}
\begin{tabular}{|c|c|c|} 
\hline
$\ell=1$ & $\ell=0$ and $k_r>0$ & $\ell=0$ and $k_r=0$\\ 
\hline
$k_e+k_b$ empty piles & $k_e+k_b$ empty piles & $k_e+k_b-1$ empty piles \\ 
$k_r$ $r$-singletons & $k_r-1$ $r$-singletons &  \\
 & & $1$ $b$-singleton \\
$1$ long $b$-pile & $1$ long $b$-pile & \\ 
\hline
\end{tabular}
\end{center}

Suppose $\mathscr{R}$ starts their round by making a sequence of $t+t'+s \geq 0$ moves, that consists of placing (refer to Theorem~\ref{thm:SameActive}):
$$
\left\{
    \begin{array}{l}
        \textrm{$t$ $b$ chips on $t$ empty piles,}\\
        \textrm{$t'$ $b$ chips on $t'$ $r$-singletons,}\\
        \textrm{and $s$ $r$ chips on $s$ $r$-singletons.}
    \end{array}
\right.
$$

The state of the board $\mathbb{B}$ is now the following, depending on the values of $\ell$ and $k_r$.
\begin{center}
\begin{tabular}{ |c|c|c| } 
 \hline
 $\ell=1$ & $\ell=0$ and $k_r>0$ & $\ell=0$ and $k_r=0$ \\ 
 \hline
 $k_e+k_b-t+s$ empty piles & $k_e+k_b-t+s$ empty piles & $k_e+k_b-t-1$ empty piles \\ 
 $k_r-t'-s$ $r$-singletons & $k_r-t'-s-1$ $r$-singletons &  \\ 
 $t$ $b$-singletons & $t$ $b$-singletons & $t+1$ $b$-singletons \\
 $t' +1$ long $b$-piles & $t' +1$ long $b$-piles & \\ 
 \hline
\end{tabular}
\end{center}

Observe that when $\ell=0$ and $k_r=0$, since there are no $r$-singletons, we have $t'=s=0$. When $\ell \neq 0$ or $k_r \neq 0$, let us denote the $t' +1$ long $b$-piles by $\beta_1,\beta_2,...,\beta_{t'+1}$, respectively.
Observe that in all three cases, there is no long $r$-pile.

Moreover, in all three cases, we have $\mathscr{R}=\left(n_b', n_r\right)$,
where $n_b' \leq n_b-t-t'$. If $n_b' = n_r=0$, then $\mathscr{R}$ is out of chips, they are eliminated and $\mathscr{B}$ wins.
Otherwise, $\mathscr{R}$ can end their round in four different ways
(refer to Theorem~\ref{thm:DifferentActive}), by:
(1) placing an $r$ chip on an empty pile,
(2) placing an $r$ chip on a $b$-singleton,
(3) placing an $r$ chip on a long $b$-pile,
or (4) placing a $b$ chip on a $b$-pile.
\begin{enumerate}
\item Suppose $\mathscr{R}$ places an $r$ chip on an empty pile. Thus, we have $n_r \geq 1$ and no long $r$-pile is created. Since $m_b > n_r $,
we then have $m_b' \geq m_b-1>n_r-1 \geq 0$, from which $m_b' > 0$. The move goes to $\mathscr{B}$ who captures all $b$-piles and discards all their prisoners, as per steps 1 and 2 of $\strat$.
Therefore, in all three cases, there is no long pile on the board.
As such, the board is of type I and it is still $\mathscr{B}$'s turn to play. Moreover,
we have $\mathscr{B}=(m_b'',0)$ and $\mathscr{R}=(n_b'',n_r'')$,
where $n_r''=n_r-1$, and $m_b''$ and $n_b''$ satisfy the following inequalities,
depending on the values of $\ell$ and $k_r$.
\begin{center}
\begin{tabular}{ |c|c|c| } 
 \hline
$\ell=1$ & $\ell=0$ and $k_r>0$ & $\ell=0$ and $k_r=0$ \\ \hline
\multicolumn{2}{|c|}{$m_b'' \geq m_b-1+ \sum_{i=1}^{t'+1}|\beta_{i}|_b$} & $m_b'' \geq m_b-1$\\ 
\multicolumn{2}{|c|}{$n_b'' \leq n_b-t-t'$} & $n_b'' \leq n_b-t-t'$\\
 \hline
\end{tabular}
\end{center}
Since $(n_b-t-t')+(n_r-1)<n_b+n_r$ and $m_b-1+\sum_{i=1}^{t'+1}|\beta_{i}|_b \geq m_b-1 > n_r-1$, 
then $m_b'' > n_r'' \geq 0$ and $n_b''+n_r'' < n_b+n_r$ in all three cases.

\item Suppose $\mathscr{R}$ places an $r$ chip on a $b$-singleton. Thus, we have $n_r \geq 1$ and one $(b,r)$-pile is created. Since $m_b > n_r $,
we then have $m_b' \geq m_b-1>n_r-1 \geq 0$, from which $m_b' > 0$. The move goes to $\mathscr{B}$, who captures all $b$-piles and discards all their prisoners, as per steps 1 and 2 of $\strat$.
Therefore, in all three cases, there is one long pile on the board, which is an $r$-pile.
As such, the board is of type I and the rest of the argument for this case is identical to the one for Case 1.

\item Suppose $\mathscr{R}$ places an $r$ chip on a long $b$-pile, say $\beta_1$ without loss of generality.
Note that we do not need to consider the case where $\ell = k_r = 0$ since there is no long $b$-pile in that case.
Thus, we have $n_r \geq 1$ and one long $r$-pile is created. 
Since $m_b > n_r $,
we then have $m_b' \geq m_b-1>n_r-1 \geq 0$, from which $m_b' > 0$. 
The move goes to $\mathscr{B}$, who captures all $b$-piles and discards all their prisoners, as per steps 1 and 2 of $\strat$.
Therefore, in all two cases, there is one long pile on the board, which is an $r$-pile.
As such, $\mathbb{B}$ is of type I and it is still $\mathscr{B}$'s turn to play. Moreover,
we have $\mathscr{B}=(m_b'',0)$ and $\mathscr{R}=(n_b'',n_r'')$,
where $n_r''=n_r-1$, and $m_b''$ and $n_b''$ satisfy the following inequalities, for all values of $\ell$ and $k_r$:
$$ m_b'' \geq m_b-1+\sum_{i=2}^{t'+1}|\beta_i|_b, \qquad
\text{and} \qquad n_b'' \leq n_b-t-t'.$$
Since $(n_b-t-t')+(n_r-1)<n_b+n_r$ and $m_b-1+\sum_{i=2}^{t'+1}|\beta_i|_b \geq m_b-1 > n_r-1$, 
then $m_b'' > n_r''\geq 0$ and $n_b''+n_r'' < n_b+n_r$ in all two cases.

\item Suppose $\mathscr{R}$ places a $b$ chip on a $b$-pile. This results in a capture for $\mathscr{B}$, from which $\mathscr{B}$ gains at least one $b$ chip.
The move goes to $\mathscr{B}$, who captures all $b$-piles and discards all their prisoners, as per steps 1 and 2 of $\strat$.
Therefore, in all three cases, there is no long pile on the board.
As such, $\mathbb{B}$ is of type I and it is still $\mathscr{B}$'s turn to play. Moreover,
we have $\mathscr{B}=(m_b'',0)$ and $\mathscr{R}=(n_b'',n_r'')$,
where $n_r''=n_r$, and $m_b''$ and $n_b''$ satisfy the following inequalities,
depending on the values of $\ell$ and $k_r$.
\begin{center}
\begin{tabular}{ |c|c|c| } 
\hline
$\ell=1$ & $\ell=0$ and $k_r>0$ & $\ell=0$ and $k_r=0$ \\ 
\hline
\multicolumn{2}{|c|}{$m_b'' \geq m_b+\sum_{i=1}^{t'+1}|\beta_i|_b$} & $m_b'' \geq m_b$\\ 
\multicolumn{2}{|c|}{$n_b'' \leq n_b-t-t'-1$} & $n_b'' \leq n_b-t-t'-1$\\
 \hline
\end{tabular}
\end{center}
Since $(n_b-t-t'-1)+n_r<n_b+n_r$ and $m_b+\sum_{i=1}^{t'+1}|\beta_i|_b \geq m_b > n_r $, 
then 
$m_b'' > n_r'' \geq 0$ and $n_b''+n_r'' < n_b+n_r$ in all three cases.
\end{enumerate}
\end{proof}

\begin{prop}
\label{prop-typeI-R}
Let $\mathbb{B}=(k_e,k_r,k_b,\ell,0)$ be of type I, $\mathscr{B}=\left(m_b, m_r\right)$, $\mathscr{R}=\left(n_b, n_r\right)$, and assume that $\mathscr{B}$ is the active player. If $0 \leq m_b \leq n_r$ when $\mathscr{B}$'s turn starts, then $\mathscr{R}$ wins by applying $\strat$ at every round.
\end{prop}

We prove Proposition~\ref{prop-typeI-R} using the following inductive argument on $m_b+m_r$.
\begin{description}
\item[Base case:] We prove that for all game states where $\mathbb{B}$ is of type I, $\mathscr{B}=\left(0, 0\right)$ and $\mathscr{R}=\left(n_b, n_r\right)$,
if $\mathscr{B}$ is the active player, then $\mathscr{R}$ wins by applying $\strat$.

\item[Inductive step:] Consider a game state where $\mathbb{B} = (k_e,k_r,k_b,\ell,0)$ is of type I, $\mathscr{B}=\left(m_b, m_r\right)$, and $\mathscr{R}=\left(n_b, n_r\right)$, such that $0\leq m_b \leq n_r$ and $m_b+m_r > 0$.
Assume $\mathscr{B}$ is the active player.
Moreover, suppose that, no matter what $\mathscr{B}$ plays, $\mathscr{R}$ applies $\strat$ during their round. Then, when $\mathscr{R}$'s round is over, we get $\mathbb{B} = (k_e'',k_r'',k_b'',\ell'',0)$, $\mathscr{B}=\left(m_b'', m_r''\right)$ and $\mathscr{R}=\left(n_b'', n_r''\right)$, where $\ell''\in\{0,1\}$, $\mathscr{B}$ is the active player, $0\leq m_b'' \leq n_r''$, and $m_b''+m_r'' < m_b+m_r$.
\end{description}

\begin{proof}
Suppose $\mathbb{B} = (k_e,k_r,k_b,\ell,0)$ is of type I.

Consider a game state where $m_b=m_r=0$ (base case). Since $\mathscr{B}$ is out of chips, they are eliminated and $\mathscr{R}$ wins.

Consider now a game state where $\mathscr{B}=\left(m_b, m_r\right)$ and $\mathscr{R}=\left(n_b, n_r\right)$ with $0 \leq m_b \leq n_r$ and $m_b+m_r > 0$.
Suppose $\mathscr{B}$ starts their round by making a sequence of $t+t'+s \geq 0$ moves, that consists of placing (refer to Theorem~\ref{thm:SameActive}):
$$
\left\{
    \begin{array}{l}
        \textrm{$t$ $r$ chips on $t$ empty piles,}\\
        \textrm{$t'$ $r$ chips on $t'$ $b$-singletons,}\\
        \textrm{and $s$ $b$ chips on $s$ $b$-singletons.}
    \end{array}
\right.
$$

The state of the board $\mathbb{B}$ is now the following: 
$$ \mathbb{B}=(k_e-t+s,k_r+t,k_b-s-t',t'+\ell,0). $$
Let us denote the $t'+\ell$ long $r$-piles by $\rho_1,\ldots,\rho_{t'+\ell}$, respectively.
Observe that there is no long $b$-pile.

We have $\mathscr{B} = (m_b',m_r')$, where $m_b'=m_b$ and $m_r' \leq m_r-t-t'$. If $m_b'=m_r'=0$, then $\mathscr{B}$ is out of chips. Therefore, $\mathscr{B}$ is eliminated and $\mathscr{R}$ wins. Otherwise,
$\mathscr{B}$ can end their round in four different ways (refer to Theorem~\ref{thm:DifferentActive}), by:
(1) placing a $b$ chip on an empty pile,
(2) placing a $b$ chip on an $r$-singleton,
(3) placing a $b$ chip on a long $r$-pile,
or (4) placing an $r$ chip on an $r$-pile.
\begin{enumerate}
\item Suppose $\mathscr{B}$ places a $b$ chip on an empty pile.
Thus, we have $m_b \geq 1$, from which $n_r \geq m_b \geq 1$. Moreover, a $b$-singleton is created and no long pile is created.
The move goes to $\mathscr{R}$, who captures all $r$-piles and discards all their prisoners, as per steps 1 and 2 of $\strat$. Then, as per Step 3 of $\strat$, $\mathscr{R}$ places an $r$ chip on a $b$-singleton.
Therefore, there is one long pile on the board, which is a $(b,r)$-pile. As such, $\mathbb{B}$ is of type I and it is $\mathscr{B}$'s turn to play.
Moreover, we have 
$\mathscr{B} = (m_b'',m_r'')$ with $m_b''=m_b-1$ and $m_r'' \leq m_r-t-t'$, and $\mathscr{R} = (0,n_r'')$ with $n_r'' \geq n_r+\sum_{i=1}^{t'+\ell}|\rho_i|_r-1$.
Therefore, we have
\begin{align*}
0 \leq m_b''&=m_b-1 \leq n_r-1 \leq n_r+\sum_{i=1}^{t'+\ell}|\rho_i|_r-1 \leq n_r'',\\
m_b''+m_r''&=m_b-1+m_r'' \leq m_b-1 + m_r-t-t' < m_b+m_r.
\end{align*}

\item Suppose $\mathscr{B}$ places a $b$ chip on an $r$-singleton.
Thus, we have $m_b \geq 1$, from which $n_r \geq m_b \geq 1$. Moreover, the long $b$-pile $(r,b)$ is created.
The move goes to $\mathscr{R}$, who captures all $r$-piles and discards all their prisoners, as per steps 1 and 2 of $\strat$. Then, as per Step 3 of $\strat$, $\mathscr{R}$ places an $r$ chip on the $(r,b)$-pile.
Therefore, there is one long pile on the board, which is an $r$-pile. As such, $\mathbb{B}$ is of type I and the rest of the argument for this case is identical to the one for Case 1.

\item Suppose $\mathscr{B}$ places a $b$ chip on a long $r$-pile, say $\rho_1$ without loss of generality.
Thus, we have $m_b \geq 1$, from which $n_r \geq m_b \geq 1$. Moreover, the long $b$-pile $\langle \rho_1,(b) \rangle$ is created.
The move goes to $\mathscr{R}$, who captures all $r$-piles and discards all their prisoners, as per steps 1 and 2 of $\strat$. Then, as per Step 3 of $\strat$, $\mathscr{R}$ places an $r$ chip on the $\langle \rho_1,(b) \rangle$-pile.
Therefore, there is one long pile on the board, which is an $r$-pile. As such, $\mathbb{B}$ is of type I and it is $\mathscr{B}$'s turn to play.
Moreover, we have $\mathscr{B} = (m_b'',m_r'')$ with $m_b''=m_b-1$ and $m_r'' \leq m_r-t-t'$, and $\mathscr{R} = (0,n_r'')$ with
$n_r'' \geq  n_r+\sum_{i=2}^{t'+\ell}|\rho_i|_r-1$.
Therefore, we have 
\begin{align*} 
0\leq m_b'' &= m_b-1 \leq n_r-1 \leq n_r+\sum_{i=2}^{t'+\ell}|\rho_i|_r-1 \leq n_r'',\\
m_b''+m_r''&= m_b-1 +m_r'' \leq m_b-1 + m_r-t-t' < m_b+m_r.
\end{align*}

\item Suppose $\mathscr{B}$ places an $r$ chip on an $r$-pile.
This results in a capture for $\mathscr{R}$, from which $\mathscr{R}$ gains at least one $r$ chip.
The move goes to $\mathscr{R}$, who captures all $r$-piles and discards all their prisoners, as per steps 1 and 2 of $\strat$. Then, as per Step 3 of $\strat$, $\mathscr{R}$ places an $r$ chip on a $b$-singleton if there is one on the board, or on an empty pile otherwise. In both cases, there is at most one long pile on the board, which is an $r$-pile. As such, $\mathbb{B}$ is of type I and it is $\mathscr{B}$'s turn to play.
Moreover, we have $\mathscr{B} = (m_b'',m_r'')$ with $m_b''=m_b$ and $m_r'' \leq m_r-t-t'-1$, and $\mathscr{R} = (0,n_r'')$ with
$n_r'' \geq  n_r+\sum_{i=1}^{t'+\ell}|\rho_i|_r$.
Therefore, we have 
\begin{align*}
0\leq m_b'' &= m_b \leq n_r \leq n_r+\sum_{i=1}^{t'+\ell}|\rho_i|_r \leq n_r'',\\
m_b''+m_r'' &= m_b+m_r'' \leq m_b + (m_r-t-t'-1) < m_b+m_r.
\end{align*}
\end{enumerate}
\end{proof}

The following theorem is a direct consequence of Propositions~\ref{prop-typeI-B} and~\ref{prop-typeI-R}.
\begin{theorem}
\label{thm-type1}
Let $\mathbb{B}=(k_e,k_r,k_b,\ell,0)$ be of type I, $\mathscr{B}=\left(m_b, m_r\right)$, $\mathscr{R}=\left(n_b, n_r\right)$, and assume that $\mathscr{B}$ is the active player. Then, $\mathscr{B}$ has a winning strategy if and only if $m_b>n_r$ when $\mathscr{B}$'s turn starts.
Whenever there is a winning strategy for $\mathscr{B}$ (respectively for $\mathscr{R}$), then $\strat$ is such a strategy.
\end{theorem}

\subsection{Generalized Type I Board}
\label{subsec:gentypeI}

We slightly extend the definition of type I boards in the following way.
\begin{definition}
The board $\mathbb{B}$ is said to be of \emph{generalized type I} if $\mathbb{B} = (k_e,k_r,k_b,\ell,0)$. If $\ell\geq 1$, we denote the $\ell$ long $r$-piles by $\rho_{1}, \rho_{2}, \ldots, \rho_{\ell}$.
\end{definition}
The structure of this subsection is the same as the one for Subsection~\ref{subsec:typeI}.
We first describe the states where there is a winning strategy for $\mathscr{B}$ on a generalized type I board (refer to Propositions~\ref{prop-gentype1-B1} and~\ref{prop-gentype1-B2}). Then we describe the states where there is a winning strategy for $\mathscr{R}$ on a generalized type I board (refer to Propositions~\ref{prop-gentype1-R1} and~\ref{prop-gentype1-R}).
We finally present the characterization of the winning states for the first player in Theorem~\ref{thm-gentype1}.

\begin{prop}
\label{prop-gentype1-B1}
Let $\mathbb{B}=(k_e,k_r,k_b,\ell,0)$ be of generalized type I, $\mathscr{B}=\left(m_b, m_r\right)$, $\mathscr{R}=\left(n_b, n_r\right)$, and assume that $\mathscr{B}$ is the active player. If $m_b>0$ and $n_r=0$ when $\mathscr{B}$'s turn starts, then $\mathscr{B}$ wins by applying $\strat$ at every round.
\end{prop}

We prove Proposition~\ref{prop-gentype1-B1} using the following inductive argument on $\ell$.
\begin{description}
\item[Base case:] We prove that for all game states where $\mathbb{B} = (k_e,k_r,k_b,\ell,0)$ with $\ell\in\{0,1\}$, $\mathscr{B}=\left(m_b, m_r\right)$ with $m_b > 0$, and $\mathscr{R}=\left(n_b, 0\right)$, if $\mathscr{B}$ is the active player, then $\mathscr{B}$ wins by applying $\strat$.
\item[Inductive step:] Consider a game state where $\mathbb{B} = (k_e,k_r,k_b,\ell,0)$ with $\ell > 1$, $\mathscr{B}=\left(m_b, m_r\right)$ with $m_b > 0$, and $\mathscr{R}=\left(n_b, 0\right)$. Assume $\mathscr{B}$ is the active player and they apply $\strat$.
Then, no matter what $\mathscr{R}$ plays during their round, when $\mathscr{R}$'s round is over,
we get $\mathbb{B} = (k_e'',k_r'',k_b'',\ell'',0)$, $\mathscr{B}=\left(m_b'', m_r''\right)$ and $\mathscr{R}=\left(n_b'', 0\right)$, where $\mathscr{B}$ is the active player, $m_b'' > 0$ and $\ell'' < \ell$.
\end{description}

\begin{proof}
Suppose $\mathbb{B} = (k_e,k_r,k_b,\ell,0)$ is of generalized type I.

Consider a game state where $\ell\in\{0,1\}$ (base case). Then $\mathbb{B}$ is of type I. Since $m_b>0=n_r$, we know from Theorem~\ref{thm-type1} that $\mathscr{B}$ wins by applying $\strat$.

Consider now a game state where $\mathscr{B}=\left(m_b, m_r\right)$ with $m_b > 0$, $\mathscr{R}=\left(n_b, 0\right)$ and $\ell > 1$.
Since $m_b>0$, then $\mathscr{B}$ has at least one $b$ chip. 
Player $\mathscr{B}$ applies Strategy $\strat$: they capture $k_b$ $b$-singletons and discard all their prisoners as per steps 1 and 2 of $\strat$. Then, as per Step 3 of $\strat$, they place a $b$ chip on the largest $r$-pile on the board. We have $\mathscr{B}=(m_b',0)$, where $m_b'\geq m_b - 1$, and $\mathscr{R}$'s round starts. The state of the board $\mathbb{B}$ is now the following: 
$$ \mathbb{B}=(k_e+k_b,k_r,0,\ell-1,1). $$
Suppose $\mathscr{R}$ starts their round by making a sequence of $t+t'+t''\geq 0$ moves, that consists of placing (refer to Theorem~\ref{thm:SameActive}):
$$
\left\{
	\begin{array}{l}
		\text{$t$ $b$ chips on $t$ empty piles,}\\
		\text{$t'$ $b$ chips on $t'$ $r$-singletons}\\
		\text{and $t''$ $b$ chips on $t''$ long $r$-piles.}\\
	\end{array}
\right.
$$
The state of the board $\mathbb{B}$ is now the following:
$$ \mathbb{B}=(k_e+k_b-t,k_r-t',t,\ell-t''-1,t'+t''+1). $$
Let us denote the $t'+t''+1$ long $b$-piles by $\beta_1,\beta_2,\ldots,\beta_{t'+t''+1}$, respectively.

We have $\mathscr{R}=(n_b',0)$, where $n_b'\leq n_b-t-t'-t''$. If $n_b'=0$, then $\mathscr{R}$ is out of chips, they are eliminated and $\mathscr{B}$ wins. Otherwise, the only way $\mathscr{R}$ can end their round is by placing a $b$ chip on a $b$-pile (refer to Theorem~\ref{thm:DifferentActive}).

Suppose $\mathscr{R}$ places a $b$ chip on a $b$-pile. This results in a capture for $\mathscr{B}$, from which $\mathscr{B}$ gains at least one $b$ chip. The move goes to $\mathscr{B}$, who captures all $b$-piles and discards all their prisoners, as per steps 1 and 2 of $\strat$. The board no longer contains any long $b$-piles, and exactly $\ell''=\ell-t''-1<\ell$ long $r$-piles. As such, $\mathbb{B}$ is of generalized type I and it is still $\mathscr{B}$'s turn to play. Moreover, we have $\mathscr{B}=(m_b'',0)$ and $\mathscr{R}=(n_b'',0)$, where
$$m_b''\geq m_b+\sum_{i=1}^{t'+t''+1}|\beta_i|_b> m_b>0.$$
\end{proof}

\begin{prop}
\label{prop-gentype1-B2}
Let $\mathbb{B} = (k_e,k_r,k_b,\ell,0)$ be of generalized type I, $\mathscr{B}=\left(m_b, m_r\right)$, $\mathscr{R}=\left(n_b, n_r\right)$, and assume that $\mathscr{B}$ is the active player. If
\begin{align}
\label{eq.prop-gentype1-B2}
m_b, n_r>0 \qquad \textrm{and} \qquad m_b>n_r+\sum_{i=1}^{\ell}\left|\rho_{i}\right|_r-\max_{1\leq i\leq \ell}\big\{|\rho_{i}|_r\big\} 
\end{align}
when $\mathscr{B}$'s turn starts, then $\mathscr{B}$ wins by applying $\strat$ at every round.
\end{prop}

Let
\begin{align}
\label{eq.def.nu}
\nu\big((\rho_1,\rho_2,...,\rho_{\ell}),n_b,n_r\big)=n_b+n_r+\sum_{i=1}^{\ell}\left(\left|\rho_{i}\right|-1\right)-\max_{1\leq i\leq \ell}\big\{\left|\rho_{i}\right|-1\big\} -1 .
\end{align}
Intuitively,
$\nu\big((\rho_1,\rho_2,...,\rho_{\ell}),n_b,n_r\big)$ corresponds to the total number of chips in $\mathscr{R}$'s possession if, after $\mathscr{B}$'s current round, $\mathscr{R}$ were to capture all $r$-piles on the board, keep all prisoners they are allowed to keep, and end their round by placing a chip on the board.
In the following proof and for the rest of this paper, we assume that an empty $\max$ is equal to $0$, for instance $\max_{1\leq i\leq 0} \big\{|\rho_{i}|_r\big\}  = 0$. We prove Proposition~\ref{prop-gentype1-B2} using the following inductive argument on $\nu\big((\rho_1,\rho_2,...,\rho_{\ell}),n_b,n_r\big)$.

\begin{description}
\item[Base case:] We prove that for all game states where $\mathbb{B}=(k_e,k_r,k_b,\ell,0)$, $\mathscr{B}=\left(m_b, m_r\right)$, $\mathscr{R}=\left(n_b, n_r\right)$, \eqref{eq.prop-gentype1-B2} is satisfied and $\nu\big((\rho_1,\rho_2,...,\rho_{\ell}),n_b,n_r\big) = 0$, if $\mathscr{B}$ is the active player, then $\mathscr{B}$ wins by applying $\strat$.
\item[Inductive step:] Consider a game state where $\mathbb{B}=(k_e,k_r,k_b,\ell,0)$, $\mathscr{B}=\left(m_b, m_r\right)$, $\mathscr{R}=\left(n_b, n_r\right)$, \eqref{eq.prop-gentype1-B2} is satisfied, and $\nu\big((\rho_1,\rho_2,...,\rho_{\ell}),n_b,n_r\big) > 0$. Assume $\mathscr{B}$ is the active player and they apply $\strat$. Then, no matter what $\mathscr{R}$ plays during their round, when $\mathscr{R}$'s round is over, we get $\mathbb{B}=(k_e'',k_r'',k_b'',\ell'',0)$, $\mathscr{B}=\left(m_b'', m_r''\right)$ and $\mathscr{R}=\left(n_b'', n_r''\right)$, where $\mathscr{B}$ is the active player, \eqref{eq.prop-gentype1-B2} is satisfied by $m_b''$, $n_r''$ and $(\rho''_1,\rho''_2,...,\rho''_{\ell''})$, and $\nu\big((\rho''_1,\rho''_2,...,\rho''_{\ell''}),n_b'',n_r''\big)<\nu\big((\rho_1,\rho_2,...,\rho_{\ell}),n_b,n_r\big)$.
\end{description}

\begin{proof}
Suppose $\mathbb{B} = (k_e,k_r,k_b,\ell,0)$ is of generalized type I, and consider the parameter $\nu\big((\rho_1,\rho_2,...,\rho_{\ell}),n_b,n_r\big)$ (refer to~\eqref{eq.def.nu}). We need the following observation for the base case. In~\eqref{eq.def.nu},
\begin{align}
\label{eq.def.n.sub}
\sum_{i=1}^{\ell}\left(\left|\rho_{i}\right|-1\right)-\max_{1\leq i\leq \ell}\big\{\left|\rho_{i}\right|-1\big\} = 0
\end{align}
whenever $\ell \in \{0,1\}$. Otherwise, if $\ell \geq 2$, then~\eqref{eq.def.n.sub} is strictly positive.

Consider a game state where $\nu\big((\rho_1,\rho_2,...,\rho_{\ell}),n_b,n_r\big) = 0$ (base case). Since $n_r>0$ by~\eqref{eq.prop-gentype1-B2}, this implies that $n_b=0$, $n_r=1$, and $\ell \in \{0,1\}$ by the previous observation. As such, $\mathbb{B}$ is of type I. Inequality~\eqref{eq.prop-gentype1-B2} becomes $m_b>n_r>0$ and so, by Theorem~\ref{thm-type1}, $\mathscr{B}$ has a winning strategy.

Consider now a game state where $\mathscr{B}=\left(m_b, m_r\right)$, $\mathscr{R}=\left(n_b, n_r\right)$ and~\eqref{eq.prop-gentype1-B2} is satisfied. Moreover, assume $\nu\big((\rho_1,\rho_2,...,\rho_{\ell}),n_b,n_r\big) > 0$. Since $m_b > 0$, $\mathscr{B}$ has at least one $b$ chip. Player $\mathscr{B}$ applies Strategy $\strat$: they capture $k_b$ $b$-singletons and discard all their prisoners as per steps 1 and 2 of $\strat$. Without loss of generality, assume the largest $r$-pile on the board is $\rho_{\ell}$. Then, as per Step 3 of $\strat$, they place a $b$ chip on $\rho_{\ell}$.
We have $\mathscr{B}=(m_b', 0)$, where $m_b' \geq m_b - 1$, and $\mathscr{R}$'s round starts. The state of the board $\mathbb{B}$ is now the following: 
$$ \mathbb{B}=(k_e+k_b,k_r,0,\ell-1,1). $$
Note that the one long $b$-pile on the board is $\langle\rho_{\ell},(b)\rangle$. Suppose $\mathscr{R}$ starts their round by making a sequence of $t+t'+t''+s+s' \geq 0$ moves, that consists of placing (refer to Theorem \ref{thm:SameActive}):
$$
\left\{
	\begin{array}{l}
		\text{$t$ $b$ chips on $t$ empty piles,}\\
		\text{$t'$ $b$ chips on $t'$ $r$-singletons,}\\
		\text{$t''$ $b$ chips on $t''$ long $r$-piles,}\\
		\text{$s$ $r$ chips on $s$ $r$-singletons}\\
		\text{and $s'$ $r$ chips on $s'$ long $r$-piles.}
	\end{array}
\right.
$$

The state of the board $\mathbb{B}$ is now the following: 
$$ \mathbb{B}=(k_e+k_b-t+s+s',k_r-t'-s,t,\ell-t''-s'-1,t'+t''+1). $$
Let us denote the $s'$ long $r$-piles that were captured by $\rho_{1}, \ldots, \rho_{s'}$ and the remaining long $r$-piles on the board by 
$\rho_{s'+1}, \ldots, \rho_{\ell-t''-1}$. The $t''+1$ long $b$-piles that were created from long $r$-piles are $\langle\rho_{\ell-t''},(b)\rangle,\langle\rho_{\ell-t''+1},(b)\rangle, \ldots,\langle\rho_{\ell},(b)\rangle$. We rename the $t'+t''+1$ long $b$-piles as $\beta_1,\beta_2,\ldots,\beta_{t'+t''+1}$.

For each pile captured by $\mathscr{R}$, one chip must be discarded (by the Capture Rule). This chip can either be a guard or a prisoner.
Therefore, we have $\mathscr{R}=(n_b', n_r')$ where
\begin{align*}
n_b' &\leq n_b-t-t'-t''+\sum_{i=1}^{s'}|\rho_i|_b ,\\
n_r' &\leq n_r+\sum_{i=1}^{s'}|\rho_i|_r , \\
n_b'+n_r' &\leq n_b+n_r-t-t'-t''+\sum_{i=1}^{s'}(|\rho_i|-1) .
\end{align*}

Player $\mathscr{R}$ can end their round in four different ways (refer to Theorem \ref{thm:DifferentActive}), by: (1) placing an $r$ chip on an empty pile, (2) placing an $r$ chip on a $b$-singleton, (3) placing an $r$ chip on a long $b$-pile, or (4) placing a $b$ chip on a $b$-pile.
\begin{enumerate}
\item Suppose $\mathscr{R}$ places an $r$ chip on an empty pile. Then, an $r$-singleton is created and no long pile is created. The move goes to $\mathscr{B}$, who captures all $b$-piles and discards all their prisoners, as per steps 1 and 2 of $\strat$. 
Therefore, there are $\ell-t''-s'-1$ long piles on the board, which are all long $r$-piles. As such, $\mathbb{B}$ is of generalized type I and it is $\mathscr{B}$'s turn to play. Moreover, we have $\mathscr{B}=(m_b'',0)$ and $\mathscr{R}=(n_b'', n_r'')$, where 
\begin{align}
\label{ineq.mbpp.3.7}
m_b''&\geq m_b-1+\sum_{i=1}^{t'+t''+1}|\beta_i|_b , \\
\label{ineq.nbpp.3.7}
n_b'' &\leq n_b-t-t'-t''+\sum_{i=1}^{s'}|\rho_i|_b , \\
\label{ineq.nrpp.3.7}
n_r'' &\leq n_r+\sum_{i=1}^{s'}|\rho_i|_r -1 , \\
\label{ineq.sumnbnrpp.3.7}
n_b''+n_r'' &\leq n_b+n_r-t-t'-t''+\sum_{i=1}^{s'}(|\rho_i|-1) - 1.
\end{align}

First, note that by~\eqref{ineq.mbpp.3.7}, we have
$$
m_b''\geq m_b-1+\sum_{i=1}^{t'+t''+1}|\beta_i|_b
\geq m_b-1+\sum_{i=1}^{1}|\beta_i|_b
\geq m_b-1+1
= m_b>0.
$$
If $n_r'' = 0$, then $\mathscr{B}$ has a winning strategy by Proposition~\ref{prop-gentype1-B1}. Otherwise, then $n_r'' > 0$ and the first part of~\eqref{eq.prop-gentype1-B2} is satisfied by $m_b''$ and $n_r''$.
Moreover, we have  
\begin{align}
\nonumber
m_b''>\: &m_b-1 & \text{by~\eqref{ineq.mbpp.3.7},}\\
\nonumber
>\: &n_r+\sum_{i=1}^{\ell}|\rho_{i}|_r-\max_{1\leq i\leq \ell}\big\{|\rho_i|_r\big\}-1 & \text{by~\eqref{eq.prop-gentype1-B2},}\\
\label{proof.prop-gentype1-B2.1}
=\: &n_r+\sum_{i=1}^{\ell-1}|\rho_{i}|_r -1 \\
\nonumber
\geq\: &n_r+\sum_{i=1}^{\ell-t''-1}|\rho_i|_r-1-\max_{s'+1\leq i\leq \ell-t''-1}\big\{|\rho_i|_r\big\}\\
\nonumber
=\: &n_r+\sum_{i=1}^{s'}|\rho_i|_r-1+\sum_{i=s'+1}^{\ell-t''-1}|\rho_i|_r-\max_{s'+1\leq i\leq \ell-t''-1}\big\{|\rho_i|_r\big\} \\
\nonumber
\geq\: &n_r''+\sum_{i=s'+1}^{\ell-t''-1}|\rho_i|_r-\max_{s'+1\leq i\leq \ell-t''-1}\big\{|\rho_i|_r\big\} & \text{by~\eqref{ineq.nrpp.3.7},}
\end{align}
where~\eqref{proof.prop-gentype1-B2.1} follows from the fact that $\rho_\ell$ was the largest $r$-pile at the beginning. Therefore, \eqref{eq.prop-gentype1-B2} is satisfied by $m_b''$, $n_r''$ and $(\rho_{s'+1},...,\rho_{\ell-t''-1})$.

Finally, we have
\begin{align}
\nonumber
& \nu\left((\rho_{s'+1},...,\rho_{\ell-t''-1}),n_b'',n_r''\right)\\
\nonumber
=\: & n_b''+n_r''+\sum_{i=s'+1}^{\ell-t''-1}(|\rho_i|-1)-\max_{s'+1\leq i\leq \ell-t''-1}\big\{|\rho_i|-1\big\}-1\\
\nonumber
\leq\: &n_b+n_r+\sum_{i=1}^{s'}(|\rho_i|-1)-1+\sum_{i=s'+1}^{\ell-t''-1}(|\rho_i|-1)-\max_{s'+1 \leq i\leq \ell-t''-1}\big\{|\rho_i|-1\big\}-1 & \text{by~\eqref{ineq.sumnbnrpp.3.7},}\\
\nonumber
=\: &n_b+n_r+\sum_{i=1}^{\ell-t''-1}(|\rho_i|-1)-\max_{s'+1 \leq i\leq \ell-t''-1}\big\{|\rho_i|-1\big\}-2\\
\nonumber
\leq\: &n_b+n_r+\sum_{i=1}^{\ell-1}(|\rho_i|-1)-2\\
\label{proof.prop-gentype1-B2.2}
<\: &n_b+n_r+\sum_{i=1}^{\ell}\left(\left|\rho_{i}\right|-1\right)-\max_{1\leq i\leq \ell}\big\{\left|\rho_{i}\right|-1\big\}-1\\
\nonumber
=\:&\nu\big((\rho_1,\rho_2,...,\rho_{\ell}),n_b,n_r\big),
\end{align}
where~\eqref{proof.prop-gentype1-B2.2} follows from the fact that $\rho_\ell$ was the largest $r$-pile at the beginning.

\item Suppose $\mathscr{R}$ places an $r$ chip on a $b$-singleton. Then, $t\geq 1$ and another long $r$-pile was created, namely a $(b,r)$-pile. The move goes to $\mathscr{B}$, who captures all $b$-piles and discards all their prisoners, as per steps 1 and 2 of $\strat$. Therefore, there are $\ell-t''-s'$ long piles on the board, which are all long $r$-piles. As such, $\mathbb{B}$ is of generalized type I and it is $\mathscr{B}$'s turn to play. Moreover, we have $\mathscr{B}=(m_b'',0)$ and $\mathscr{R}=(n_b'', n_r'')$, where~\eqref{ineq.mbpp.3.7}, \eqref{ineq.nbpp.3.7}, \eqref{ineq.nrpp.3.7} and~\eqref{ineq.sumnbnrpp.3.7} are satisfied. 

As in Case 1, we have $m_b''>0$. If $n_r'' = 0$, then $\mathscr{B}$ has a winning strategy by Proposition~\ref{prop-gentype1-B1}. Otherwise, then $n_r'' > 0$ and the first part of~\eqref{eq.prop-gentype1-B2} is satisfied by $m_b''$ and $n_r''$. Moreover, we have
\begin{align}
\nonumber
m_b''>\: &m_b-1 & \text{by~\eqref{ineq.mbpp.3.7},}\\
\nonumber
>\: &n_r+\sum_{i=1}^{\ell}|\rho_i|_r-\max_{1\leq i\leq \ell}\big\{|\rho_i|_r\big\}-1 & \text{by~\eqref{eq.prop-gentype1-B2},}\\
\label{proof.prop-gentype1-B2.3}
=\: &n_r+\sum_{i=1}^{\ell-1}|\rho_i|_r-1\\
\nonumber
\geq\: &n_r +\sum_{i=1}^{\ell-t''-1}|\rho_i|_r-1 +|(b,r)|_r-\max\left\{|(b,r)|_r,\max_{s'+1\leq i\leq \ell-t''-1}\big\{|\rho_i|_r\big\}\right\}\\
\nonumber
=\: &n_r+\sum_{i=1}^{s'}|\rho_i|_r-1+\sum_{i=s'+1}^{\ell-t''-1}|\rho_{i}|_r+|(b,r)|_r-\max\left\{|(b,r)|_r,\max_{s'+1\leq i\leq \ell-t''-1}\big\{|\rho_i|_r\big\}\right\} \\
\nonumber
\geq\: &n_r''+\sum_{i=s'+1}^{\ell-t''-1}|\rho_{i}|_r+|(b,r)|_r-\max\left\{|(b,r)|_r,\max_{s'+1\leq i\leq \ell-t''-1}\big\{|\rho_i|_r\big\}\right\} & \text{by~\eqref{ineq.nrpp.3.7},}
\end{align}
where~\eqref{proof.prop-gentype1-B2.3} follows from the fact that $\rho_\ell$ was the largest $r$-pile at the beginning. Therefore, \eqref{eq.prop-gentype1-B2} is satisfied by $m_b''$, $n_r''$ and $(\rho_{s'+1},...,\rho_{\ell-t''-1},(b,r))$. Finally, we have
\begin{align}
\nonumber
& \nu\big((\rho_{s'+1},...,\rho_{\ell-t''-1},(b,r)),n_b'',n_r''\big)\\
\nonumber
=\: & n_b''+n_r''+\sum_{i=s'+1}^{\ell-t''-1}(|\rho_{i}|-1)+(|(b,r)|-1)
-\max\left\{|(b,r)|-1,\max_{s'+1\leq i\leq \ell-t''-1}\big\{|\rho_i|-1\big\}\right\}-1\\
\nonumber
\leq\: &n_b+n_r+\sum_{i=1}^{s'}(|\rho_i|-1)-1 +\sum_{i=s'+1}^{\ell-t''-1}(|\rho_{i}|-1)
+(|( b,r)|-1) \\
\nonumber
&-\max\left\{|(b,r)|-1,\max_{s'+1\leq i \leq \ell-t''-1}\big\{|\rho_i|-1\big\}\right\}-1 & \text{by~\eqref{ineq.sumnbnrpp.3.7},}\\
\nonumber
=\: &n_b+n_r+\sum_{i=1}^{\ell-t''-1}(|\rho_i|-1)+(|( b,r)|-1)-\max\left\{|(b,r)|-1,\max_{s'+1\leq i\leq \ell-t''-1}\big\{|\rho_i|-1\big\}\right\}-2\\
\nonumber
\leq\: &n_b+n_r+\sum_{i=1}^{\ell-1}(|\rho_i|-1)-2\\
\label{proof.prop-gentype1-B2.4}
<\: &n_b+n_r+\sum_{i=1}^{\ell}(|\rho_i|-1)-\max_{1\leq i\leq \ell}\big\{|\rho_i|-1\big\}-1\\
\nonumber
=\: &\nu\big((\rho_1,\rho_2,...,\rho_{\ell}),n_b,n_r\big),
\end{align}
where~\eqref{proof.prop-gentype1-B2.4} follows from the fact that $\rho_\ell$ was the largest $r$-pile at the beginning.

\item Suppose $\mathscr{R}$ places an $r$ chip on a long $b$-pile, say $\beta_1$ without loss of generality. Then, another long $r$-pile was created, namely a $\langle\beta_1,(r)\rangle$-pile. The move goes to $\mathscr{B}$, who captures all $b$-piles and discards all their prisoners, as per steps 1 and 2 of $\strat$.
Therefore, there are $\ell-t''-s'$ long piles on the board, which are all long $r$-piles. As such, $\mathbb{B}$ is of generalized type I and it is $\mathscr{B}$'s turn to play. Moreover, we have $\mathscr{B}=(m_b'',0)$ and $\mathscr{R}=(n_b'',n_r'')$, where
\begin{align}
\label{ineq.mbpp.3.7.case3}
m_b''&\geq m_b-1+\sum_{i=2}^{t'+t''+1}|\beta_i|_b, \\
\label{ineq.nbpp.3.7.case3}
n_b'' &\leq n_b-t-t'-t''+\sum_{i=1}^{s'}|\rho_i|_b , \\
\label{ineq.nrpp.3.7.case3}
n_r'' &\leq n_r+\sum_{i=1}^{s'}|\rho_i|_r -1 , \\
\label{ineq.sumnbnrpp.3.7.case3}
n_b''+n_r'' &\leq n_b+n_r-t-t'-t''+\sum_{i=1}^{s'}(|\rho_i|-1) - 1.
\end{align}

First, note that by~\eqref{eq.prop-gentype1-B2} and~\eqref{ineq.mbpp.3.7.case3}, we have
$$
m_b''\geq m_b-1+\sum_{i=2}^{t'+t''+1}|\beta_i|_b
\geq m_b-1 > 0.
$$
If $n_r'' = 0$, then $\mathscr{B}$ has a winning strategy by Proposition~\ref{prop-gentype1-B1}. Otherwise, then $n_r'' > 0$ and the first part of~\eqref{eq.prop-gentype1-B2} is satisfied by $m_b''$ and $n_r''$. Moreover, we have
\begin{align}
\nonumber
m_b''\geq\: &m_b-1 & \text{by~\eqref{ineq.mbpp.3.7.case3},} \\
\nonumber
>\: &n_r+\sum_{i=1}^{\ell}|\rho_i|_r-\max_{1\leq i\leq \ell}\big\{|\rho_i|_r\big\}-1 & \text{by~\eqref{eq.prop-gentype1-B2},}\\
\label{proof.prop-gentype1-B2.5}
=\: &n_r+\sum_{i=1}^{\ell-1}|\rho_i|_r-1\\
\nonumber
\geq\: &n_r+\sum_{i=1}^{\ell-t''-1}|\rho_i|_r-1+|\langle \beta_1,(r)\rangle|_r-\max\left\{|\langle \beta_1,(r)\rangle|_r,\max_{s'+1\leq i\leq \ell-t''-1}\big\{|\rho_i|_r\big\}\right\}\\
\nonumber
=\: &n_r+\sum_{i=1}^{s'}|\rho_i|_r-1+\sum_{i=s'+1}^{\ell-t''-1}|\rho_i|_r+|\langle \beta_1,(r)\rangle|_r
-\max\left\{|\langle \beta_1,(r)\rangle|_r,\max_{s'+1\leq i\leq \ell-t''-1}\big\{|\rho_i|_r\big\}\right\} \\
\nonumber
\geq\: &n_r''+\sum_{i=s'+1}^{\ell-t''-1}|\rho_i|_r+|\langle \beta_1,(r)\rangle|_r
-\max\left\{|\langle \beta_1,(r)\rangle|_r,\max_{s'+1\leq i\leq \ell-t''-1}\big\{|\rho_i|_r\big\}\right\} & \text{by~\eqref{ineq.nrpp.3.7.case3},}
\end{align}
where~\eqref{proof.prop-gentype1-B2.5} follows from the fact that $\rho_\ell$ was the largest $r$-pile at the beginning.
Therefore, \eqref{eq.prop-gentype1-B2} is satisfied by $m_b''$, $n_r''$ and $(\rho_{s'+1},...,\rho_{\ell-t''-1},\langle\beta_1,(r)\rangle)$. Finally, we have
\begin{align}
\nonumber
& \nu\big((\rho_{s'+1},...,\rho_{\ell-t''-1},\langle\beta_1,(r)\rangle),n_b'',n_r''\big)\\
\nonumber
=\: & n_b''+n_r''+\sum_{i=s'+1}^{\ell-t''-1}\left(\left|\rho_{i}\right|-1\right) + (|\langle \beta_1,(r)\rangle|-1) \\
\nonumber
&-\max\left\{|\langle \beta_1,(r)\rangle|-1,\max_{s'+1\leq i\leq \ell-t''-1}\big\{|\rho_i|-1\big\}\right\}-1\\
\nonumber
\leq\: &n_b+n_r+\sum_{i=1}^{s'}(|\rho_i|-1)-1+\sum_{i=s'+1}^{\ell-t''-1}\left(\left|\rho_{i}\right|-1\right) + (|\langle\beta_1,(r)\rangle|-1) & \text{by~\eqref{ineq.sumnbnrpp.3.7.case3},}\\
\nonumber
&-\max\left\{|\langle\beta_1,(r)\rangle|-1,\max_{s'+1\leq i\leq \ell-t''-1}\big\{|\rho_i|-1\big\}\right\}-1\\
\nonumber
=\: &n_b+n_r+\sum_{i=1}^{\ell-t''-1}\left(\left|\rho_{i}\right|-1\right)+ (|\langle\beta_1,(r)\rangle|-1)\\
\nonumber
&-\max\left\{|\langle\beta_1,(r)\rangle|-1,\max_{s'+1\leq i\leq \ell-t''-1}\big\{|\rho_i|-1\big\}\right\}-2\\
\nonumber
\leq\: &n_b+n_r+\sum_{i=1}^{\ell-1}\left(\left|\rho_{i}\right|-1\right)-2\\
\label{proof.prop-gentype1-B2.6}
<\: &n_b+n_r+\sum_{i=1}^{\ell}\left(\left|\rho_{i}\right|-1\right)-\max_{1\leq i\leq \ell}\big\{|\rho_i|-1\big\}-1\\
\nonumber
=\: &\nu\big((\rho_1,\rho_2,...,\rho_{\ell}),n_b,n_r\big),
\end{align}
where~\eqref{proof.prop-gentype1-B2.6} follows from the fact that $\rho_\ell$ was the largest $r$-pile at the beginning.

\item Suppose $\mathscr{R}$ places a $b$ chip on a $b$-pile. This results in a capture for $\mathscr{B}$, from which $\mathscr{B}$ gains at least one $b$ chip. The move goes to $\mathscr{B}$, who captures all $b$-piles and discards all their prisoners, as per steps 1 and 2 of $\strat$. Therefore, there are $\ell-s'-t''-1$ long piles on the board, which are all long $r$-piles. As such, $\mathbb{B}$ is of generalized type I and it is $\mathscr{B}$'s turn to play. Moreover, we have $\mathscr{B}=(m_b'',0)$ and $\mathscr{R}=(n_b'',n_r'')$, where
\begin{align}
\label{ineq.mbpp.3.7.case4}
m_b''&\geq m_b+\sum_{i=1}^{t'+t''+1}|\beta_i|_b, \\
\label{ineq.nbpp.3.7.case4}
n_b'' &\leq n_b-t-t'-t''+\sum_{i=1}^{s'}|\rho_i|_b-1 , \\
\label{ineq.nrpp.3.7.case4}
n_r'' &\leq n_r+\sum_{i=1}^{s'}|\rho_i|_r , \\
\label{ineq.sumnbnrpp.3.7.case4}
n_b''+n_r'' &\leq n_b+n_r-t-t'-t''+\sum_{i=1}^{s'}(|\rho_i|-1) - 1.
\end{align}

First, note that by~\eqref{ineq.mbpp.3.7.case4}, we have
$$m_b''\geq\: m_b+\sum_{i=1}^{t'+t''+1}|\beta_i|_b > m_b > 0 .$$
If $n_r'' = 0$, then $\mathscr{B}$ has a winning strategy by Proposition~\ref{prop-gentype1-B1}. Otherwise, then $n_r'' > 0$ and the first part of~\eqref{eq.prop-gentype1-B2} is satisfied by $m_b''$ and $n_r''$. Moreover, we have
\begin{align}
\nonumber
m_b''>\: &m_b & \text{by~\eqref{ineq.mbpp.3.7.case4},}\\
\nonumber
>\: &n_r+\sum_{i=1}^{\ell}|\rho_i|_r-\max_{1\leq i\leq \ell}\big\{|\rho_i|_r\big\} & \text{by~\eqref{eq.prop-gentype1-B2},}\\
\label{proof.prop-gentype1-B2.7}
=\: &n_r+\sum_{i=1}^{\ell-1}|\rho_i|_r\\
\nonumber
\geq\: &n_r+\sum_{i=1}^{\ell-t''-1}|\rho_i|_r-\max_{s'+1\leq i \leq \ell-t''-1}\big\{|\rho_i|_r\big\}\\
\nonumber
=\: &n_r+\sum_{i=1}^{s'}|\rho_i|_r+ \sum_{i=s'+1}^{\ell-t''-1}|\rho_i|_r-\max_{s'+1\leq i\leq \ell-t''-1}\big\{|\rho_i|_r\big\} \\
\nonumber
\geq\: &n_r'' + \sum_{i=s'+1}^{\ell-t''-1}|\rho_i|_r-\max_{s'+1\leq i\leq \ell-t''-1}\big\{|\rho_i|_r\big\} & \text{by~\eqref{ineq.nrpp.3.7.case4},}
\end{align}
where~\eqref{proof.prop-gentype1-B2.7} follows from the fact that $\rho_\ell$ was the largest $r$-pile at the beginning.
Therefore, \eqref{eq.prop-gentype1-B2} is satisfied by $m_b''$, $n_r''$ and $(\rho_{s'+1},...,\rho_{\ell-t''-1})$.

Finally, we have
\begin{align}
\nonumber
& \nu\big((\rho_{s'+1},...,\rho_{\ell-t''-1}),n_b'',n_r''\big)\\
\nonumber
=\: & n_b''+n_r''+\sum_{i=s'+1}^{\ell-t''-1}(|\rho_i|-1)-\max_{s'+1\leq i\leq \ell-t''-1}\big\{|\rho_i|-1\big\}-1\\
\nonumber
\leq\: &n_b+n_r+\sum_{i=1}^{s'}(|\rho_i|-1)-1+\sum_{i=s'+1}^{\ell-t''-1}(|\rho_i|-1)-\max_{s'+1\leq i\leq \ell-t''-1}\big\{|\rho_i|-1\big\}-1 & \text{by~\eqref{ineq.sumnbnrpp.3.7.case4},} \\
\nonumber
=\: &n_b+n_r+\sum_{i=1}^{\ell-t''-1}(|\rho_i|-1)-\max_{s'+1\leq i\leq \ell-t''-1}\big\{|\rho_i|-1\big\}-2\\
\nonumber
\leq\: &n_b+n_r+\sum_{i=1}^{\ell-1}(|\rho_i|-1)-2\\
\label{proof.prop-gentype1-B2.8}
<\: &n_b+n_r+\sum_{i=1}^{\ell}(|\rho_i|-1)-\max_{1\leq i\leq \ell}\big\{|\rho_i|-1\big\}-1\\
\nonumber
=\: &\nu\big((\rho_1,\rho_2,...,\rho_{\ell}),n_b,n_r\big),
\end{align}
where~\eqref{proof.prop-gentype1-B2.8} follows from the fact that $\rho_\ell$ was the largest $r$-pile at the beginning.
\end{enumerate}
\end{proof}

\begin{prop}
\label{prop-gentype1-R1}
Let $\mathbb{B}= (k_e,k_r,k_b,\ell,0)$ be of generalized type I, $\mathscr{B}=\left(m_b, m_r\right)$, $\mathscr{R}=\left(n_b, n_r\right)$, and assume that $\mathscr{B}$ is the active player. If $m_b=0$
when $\mathscr{B}$'s turn starts, then $\mathscr{R}$ wins by applying $\strat$ at every round.
\end{prop}

\begin{proof}
Suppose $\mathbb{B}=(k_e,k_r,k_b,\ell,0)$ is of generalized type I. If $\ell\in\{0,1\}$, then $\mathscr{R}$ wins by Proposition~\ref{prop-typeI-R}. For the rest of the proof, assume $\ell \geq 2$.

If $m_b=m_r=0$, then $\mathscr{B}$ is out of chips, they are eliminated and $\mathscr{R}$ wins. Otherwise, $\mathscr{B}$ starts their round by making a sequence of $t+t'\geq 0$ moves, that consist of placing (refer to Theorem~\ref{thm:SameActive}):
$$
\left\{
	\begin{array}{l}
		\text{$t$ $r$ chips on $t$ empty piles,}\\
		\text{and $t'$ $r$ chips on $t'$ $b$-singletons.}
	\end{array}
\right.
$$
The state of the board $\mathbb{B}$ is now the following:
$$\mathbb{B}=(k_e-t,k_r+t,k_b-t',\ell+t',0). $$
Let us denote the $t'$ $(b,r)$-piles that were created by $\rho_{\ell+1},\rho_{\ell+2},\ldots,\rho_{\ell+t'}$, respectively.

We have $\mathscr{B}=(0,m_r')$, where $m_r'\leq m_r-t-t'$. If $m_r'=0$, then $\mathscr{B}$ is out of chips, they are eliminated and $\mathscr{R}$ wins. Otherwise, $\mathscr{B}$ must end their round by placing an $r$ chip on an $r$-pile (refer to Theorem~\ref{thm:DifferentActive}).

Suppose $\mathscr{B}$ places an $r$ chip on an $r$-pile. This results in a capture for $\mathscr{R}$, from which $\mathscr{R}$ gains at least one $r$ chip. The move goes to $\mathscr{R}$, who captures all $r$-piles and discards all their prisoners, as per steps 1 and 2 of $\strat$. Then, as per Step 3 of $\strat$, $\mathscr{R}$ places an $r$ chip on a $b$-singleton if there is one, or on an empty pile otherwise. In both cases, $\mathbb{B}$ is of type I and it is $\mathscr{B}$'s turn to play. Moreover, we have $\mathscr{B}=(0, m_r'')$ and $\mathscr{R}=(0, n_r'')$ where 
$$m_r''\leq m_r-t-t'-1 \qquad \textrm{and} \qquad n_r'' \geq n_r+\sum_{i=1}^{\ell+t'}|\rho_i|_r.$$
Since
$$n_r''\geq n_r+\sum_{i=1}^{\ell+t'}|\rho_i|_r \geq 0=m_b,$$
then, by Theorem~\ref{thm-type1}, $\mathscr{R}$ wins by applying $\strat$ at every round.
\end{proof}

\begin{prop}
\label{prop-gentype1-R}
Let $\mathbb{B}= (k_e,k_r,k_b,\ell,0)$ be of generalized type I, $\mathscr{B}=\left(m_b, m_r\right)$, $\mathscr{R}=\left(n_b, n_r\right)$, and assume that $\mathscr{B}$ is the active player. If
\begin{align}
\label{eq.prop-gentype1-R}
n_r>0 \qquad \textrm{and} \qquad m_b \leq n_r+\sum_{i=1}^{\ell}\left|\rho_{i}\right|_r-\max_{1\leq i\leq \ell}\big\{\left|\rho_{i}\right|_r\big\}
\end{align}
when $\mathscr{B}$'s turn starts, then $\mathscr{R}$ wins by applying $\strat$ at every round.
\end{prop}

\begin{proof}
Suppose $\mathbb{B}=(k_e,k_r,k_b,\ell,0)$ is of generalized type I.

If $m_b=m_r=0$, then $\mathscr{B}$ is out of chips, they are eliminated and $\mathscr{R}$ wins. Otherwise, $\mathscr{B}$ starts their round by making a sequence of $t+t'+s\geq 0$ moves, that consist of placing (refer to Theorem~\ref{thm:SameActive}):
$$
\left\{
	\begin{array}{l}
		\text{$t$ $r$ chips on $t$ empty piles,}\\
		\text{$t'$ $r$ chips on $t'$ $b$-singletons}\\
		\text{and $s$ $b$ chips on $s$ $b$-singletons.}
	\end{array}
\right.
$$

The state of the board $\mathbb{B}$ is now the following:
$$\mathbb{B}=(k_e-t+s,k_r+t,k_b-t'-s,\ell+t',0). $$
Let us denote the $t'$ $(b,r)$-piles that were created by $\rho_{\ell+1},\rho_{\ell+2},\ldots,\rho_{\ell+t'}$, respectively.

We have $\mathscr{B}=(m_b,m_r')$, where $m_r'\leq m_r-t-t'$. If $m_b=m_r'=0$, then $\mathscr{B}$ is out of chips, they are eliminated and $\mathscr{R}$ wins. Otherwise, $\mathscr{B}$ can end their round in four different ways (refer to Theorem \ref{thm:DifferentActive}): (1) placing a $b$ chip on an empty pile, (2) placing a $b$ chip on an $r$-singleton, (3) placing a $b$ chip on a long $r$-pile, or (4) placing an $r$ chip on an $r$-pile.
\begin{enumerate}
\item Suppose $\mathscr{B}$ places a $b$ chip on an empty pile. Then, a $b$-singleton is created and no long pile is created. The move goes to $\mathscr{R}$, who captures all $r$-piles and discards all their prisoners, as per steps 1 and 2 of $\strat$. Then, as per Step 3 of $\strat$, $\mathscr{R}$ places $r$ chip on a $b$-singleton. Therefore, there is only one long pile on the board, which is a $(b,r)$-pile. As such, $\mathbb{B}$ is of type I and it is $\mathscr{B}$'s turn to play. Moreover, we have $\mathscr{B}=(m_b'',m_r'')$ and $\mathscr{R}=(0,n_r'')$ where $m_b'' = m_b-1$, and
$$ m_r'' \leq m_r-t-t',\qquad \textrm{and} \qquad n_r'' \geq n_r+\sum_{i=1}^{\ell+t'}|\rho_i|_r-1.$$
Since 
$$n_r''\geq n_r+\sum_{i=1}^{\ell+t'}|\rho_i|_r-1
\geq n_r+\sum_{i=1}^{\ell}|\rho_{i}|_r-1
\geq n_r+\sum_{i=1}^{\ell}|\rho_i|_r-\max_{1\leq i\leq \ell}\big\{\left|\rho_{i}\right|_r\big\}
\geq m_b
> m_b''
$$
by~\eqref{eq.prop-gentype1-R}, then, by Theorem~\ref{thm-type1}, $\mathscr{R}$ wins by applying $\strat$ at every round.

\item Suppose $\mathscr{B}$ places a $b$ chip on an $r$-singleton. Thus, a long $b$-pile is created, namely an $(r,b)$-pile. The move goes to $\mathscr{R}$, who captures all $r$-piles and discards all their prisoners, as per steps 1 and 2 of $\strat$. Then, as per Step 3 of $\strat$, $\mathscr{R}$ places an $r$ chip on the only long $b$-pile on the board, which is an $(r,b)$-pile. The board contains exactly one long pile, which is an $(r,b,r)$-pile. As such, $\mathbb{B}$ is of type I and it is $\mathscr{B}$'s turn to play. The rest of the argument for this case is identical to the one for Case 1.

\item Suppose $\mathscr{B}$ places a $b$ chip on a long $r$-pile, say $\rho_1$ without loss of generality. Thus, the long $b$-pile $\langle \rho_1,(b)\rangle$ is created. The move goes to $\mathscr{R}$, who captures all $r$-piles and discards all their prisoners, as per steps 1 and 2 of $\strat$. Then, as per Step 3 of $\strat$, $\mathscr{R}$ places an $r$ chip on the long $b$-pile $\langle \rho_,(b)\rangle$. The board contains exactly one long pile, which is a long $r$-pile. As such, $\mathbb{B}$ is of type I and it is $\mathscr{B}$'s turn to play. Moreover, we have $\mathscr{B}=(m_b'',m_r'')$ and $\mathscr{R}=(0,n_r'')$ where $m_b''=m_b-1$, and
$$m_r''\leq m_r-t-t' \qquad \textrm{and} \qquad n_r'' \geq n_r+\sum_{i=2}^{\ell+t'}|\rho_i|_r-1.$$
Since
\begin{align*}
n_r''\geq\: &n_r+\sum_{i=2}^{\ell+t'}|\rho_i|_r-1 \\
\geq\: &n_r+\sum_{i=1}^{\ell}|\rho_i|_r-\max_{1\leq i\leq \ell+t'}\big\{\left|\rho_{i}\right|_r\big\} -1 & \text{since $\max_{1\leq i\leq \ell+t''}\big\{\left|\rho_{i}\right|_r\big\} \geq |\rho_1|_r$,}\\
=\: &n_r+\sum_{i=1}^{\ell}|\rho_i|_r-\max\left\{\max_{1\leq i\leq \ell}\big\{\left|\rho_{i}\right|_r\big\},\max_{\ell+1\leq i\leq \ell+t'}\big\{\left|\rho_{i}\right|_r\big\}\right\} -1 \\
=\: &n_r+\sum_{i=1}^{\ell}|\rho_i|_r-\max\left\{\max_{1\leq i\leq \ell}\big\{\left|\rho_{i}\right|_r\big\},1\right\} -1 \\
=\: &n_r+\sum_{i=1}^{\ell}|\rho_i|_r-\max_{1\leq i\leq \ell}\big\{\left|\rho_{i}\right|_r\big\} -1 \\
\geq\: &m_b -1 & \text{by~\eqref{eq.prop-gentype1-R},}\\
=\: &m_b'',
\end{align*}
then, by Theorem~\ref{thm-type1}, $\mathscr{R}$ wins by applying $\strat$ at every round.

\item Suppose $\mathscr{B}$ places an $r$ chip on an $r$-pile. This results in a capture for $\mathscr{R}$, from which $\mathscr{R}$ gains at least one $r$ chip. The move goes to $\mathscr{R}$, who captures all $r$-piles and discards all their prisoners, as per steps 1 and 2 of $\strat$. Then, as per Step 3 of $\strat$, $\mathscr{R}$ places an $r$ chip on a $b$-singleton if there is one, or on an empty pile otherwise. In both cases, $\mathbb{B}$ is of type I and it is $\mathscr{B}$'s turn to play. Moreover, we have $\mathscr{B}=(m_b'', m_r'')$ and $\mathscr{R}=(0, n_r'')$ where $m_b''=m_b$, and
$$m_r''\leq m_r-t-t'-1 \qquad \textrm{and} \qquad n_r'' \geq n_r+\sum_{i=1}^{\ell+t'}|\rho_i|_r.$$
Since
$$n_r''\geq n_r+\sum_{i=1}^{\ell+t'}|\rho_i|_r \geq n_r+\sum_{i=1}^{\ell}|\rho_i|_r - \max_{1\leq i\leq \ell}\big\{\left|\rho_{i}\right|_r\big\} \geq m_b = m_b''$$
by~\eqref{eq.prop-gentype1-R}, then, by Theorem~\ref{thm-type1}, $\mathscr{R}$ wins by applying $\strat$ at every round.
\end{enumerate}
\end{proof}

The following theorem is a direct consequence of Propositions~\ref{prop-gentype1-B1}, \ref{prop-gentype1-B2}, \ref{prop-gentype1-R1}, and~\ref{prop-gentype1-R}.
\begin{theorem}
\label{thm-gentype1}
Let $\mathbb{B}= (k_e,k_r,k_b,\ell,0)$ be of generalized type I, $\mathscr{B}=\left(m_b, m_r\right)$, $\mathscr{R}=\left(n_b, n_r\right)$, and assume that $\mathscr{B}$ is the active player. Then, $\mathscr{B}$ has a winning strategy if and only if  
$$m_b>0 \qquad \textrm{and} \qquad \left(n_r=0 \quad \textrm{or} \quad m_b>n_r+\sum_{i=1}^{\ell}\left|\rho_{i}\right|_r-\max_{1\leq i\leq \ell}\big\{\left|\rho_{i}\right|_r\big\}\right)$$
when $\mathscr{B}$'s turn starts. Whenever there is a winning strategy for $\mathscr{B}$ (respectively for $\mathscr{R}$), then $\strat$ is such a strategy.
\end{theorem}

\section{Boards with Exactly One Long $r$-Pile}
\label{sec:GTII}

In this section, we analyze game states with two players and two colors when the board contains exactly one long $r$-pile. The structure of this section is similar to that of Section~\ref{sec:GTI}.

\subsection{Type II Board}
\label{subsec:typeII}

We start by studying game states where the only long piles the board contains are one long $r$-pile and one long $b$-pile.

\begin{definition}
\label{def:typeII}
The board $\mathbb{B}$ is said to be of \emph{type II} if $\mathbb{B}=(k_e,k_r,k_b,1,1)$ (refer to Figure~\ref{fig:type_ii}). We denote the long $r$-pile by $\rho$ and the long $b$-pile by $\beta$.
\end{definition}

\begin{figure}[H]
     \centering
     \begin{subfigure}[h]{0.4\textwidth}
         \centering
         \includegraphics[scale=1]{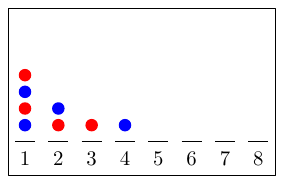}
         \label{fig:type_ii_1}
     \end{subfigure}
     \hfill
     \begin{subfigure}[h]{0.4\textwidth}
         \centering
         \includegraphics[scale=1]{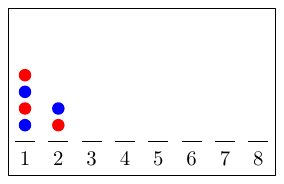}
         \label{fig:type_ii_2}
     \end{subfigure}
        \caption{Examples of type II boards with $k=8$}
        \label{fig:type_ii}
\end{figure}

The structure of this subsection is the same as the one for Subsection~\ref{subsec:gentypeI}. We first describe the states where there is a winning strategy for $\mathscr{B}$ on a type II board (refer to Proposition~\ref{prop-typeII-B}). Then we describe the states where there is a winning strategy for $\mathscr{R}$ on a type II board (refer to Corollary~\ref{cor-typeII-R}). We finally present the characterization of the winning states for the first player in Theorem~\ref{thm-type2}.

\begin{prop}
\label{prop-typeII-B}
Let $\mathbb{B}=(k_e,k_r,k_b,1,1)$ be of type II, $\mathscr{B}=\left(m_b, m_r\right)$,  $\mathscr{R}=\left(n_b, n_r\right)$, and assume that $\mathscr{B}$ is the active player. If $m_b >0$ and $m_b+|\beta|_b>n_r$ when $\mathscr{B}$'s turn starts, then $\mathscr{B}$ wins by applying $\strat$ at every round.
\end{prop}

\begin{proof}
Suppose $\mathbb{B}=(k_e,k_r,k_b,1,1)$ is of type II.

Since $m_b > 0$, then $\mathscr{B}$ has at least one $b$ chip. Player $\mathscr{B}$ applies Strategy $\strat$: they capture
all $k_b+1$ $b$-piles and discard all their prisoners as per steps 1 and 2 of $\strat$.
Therefore, there is only one long pile on the board, which is the long $r$-pile $\rho$. As such, $\mathbb{B}$ is of type I and it is $\mathscr{B}$'s turn to play.
Moreover, we have $\mathscr{B}=(m_b', 0)$ and $\mathscr{R}=(n_b', n_r')$ where $m_b'\geq m_b+|\beta|_b$, $n_b' \leq n_b$ and $n_r'=n_r$.
Since
$$m_b'\geq m_b+|\beta|_b>n_r=n_r',$$
then, by Theorem~\ref{thm-type1}, $\mathscr{B}$ wins by applying $\strat$ at every round.
\end{proof}

The goal for the rest of this section is to prove that $\mathscr{R}$ can win when $\mathbb{B}$ is of type II and, $m_b = 0$ or $m_b+|\beta|_b \leq n_r$
(refer to Corollary~\ref{cor-typeII-R}).
To prove this, we break down the argument into the following cases and sub-cases:
\begin{itemize}
\item $m_b = 0$ (refer to Proposition~\ref{prop-typeII-R1}), or
\item $m_b \geq 0$ and $m_b+|\beta|_b \leq n_r$, and
\begin{itemize}
\item $\mathscr{B}$ places a chip on $\beta$ during their first round (refer to Proposition~\ref{prop-typeII-R2}), 
\item or $\mathscr{B}$ does not place any chip on $\beta$ during their first round (refer to Proposition~\ref{prop-typeII-R3}). This case relies on all previous cases.
\end{itemize}
\end{itemize}

\begin{prop}
\label{prop-typeII-R1}
Let $\mathbb{B}=(k_e,k_r,k_b,1,1)$ be of type II, $\mathscr{B}=\left(m_b, m_r\right)$, $\mathscr{R}=\left(n_b, n_r\right)$, and assume that
$\mathscr{B}$ is the active player. If $m_b = 0$ when $\mathscr{B}$'s turn starts, then $\mathscr{R}$ wins by applying $\strat$ at every round.
\end{prop}

\begin{proof}
Suppose $\mathbb{B}=(k_e,k_r,k_b,1,1)$ is of type II.

Suppose $m_b=0$. If $m_r=0$, then $\mathscr{B}$ is out of chips, they are eliminated and $\mathscr{R}$ wins. Otherwise, $\mathscr{B}$ starts their round by making a sequence of $t+t'+t''\geq 0$ moves, that consist of placing (refer to Theorem~\ref{thm:SameActive}):
$$
\left\{
    \begin{array}{l}
        \textrm{$t$ $r$ chips on $t$ empty piles,}\\
        \textrm{$t'$ $r$ chips on $t'$ $b$-singletons,}\\
        \textrm{and $t''\in\{0,1\}$ $r$ chips on $\beta$.}
    \end{array}
\right.
$$
The state of the board $\mathbb{B}$ is now the following:
$$\mathbb{B}=(k_e-t,k_r+t,k_b-t',t'+t''+1,1-t''). $$
Let us denote the $t'+t''+1$ long $r$-piles by $\rho_1,\rho_2,...,\rho_{t'+t''+1}$, respectively.

We have $\mathscr{B}=(0,m_r')$ where $m_r'\leq m_r-t-t'-t''$. If $m_r'=0$, then $\mathscr{B}$ is out of chips, they are eliminated and $\mathscr{R}$ wins. Otherwise, $\mathscr{B}$ can only end their round by placing an $r$ chip on an $r$-pile (refer to Theorem \ref{thm:DifferentActive}). This results in a capture for $\mathscr{R}$, from which $\mathscr{R}$ gains at least one $r$ chip. The move goes to $\mathscr{R}$, who captures all $r$-piles and discards all their prisoners, as per steps 1 and 2 of $\strat$. Now, the board no longer contains any long $r$-piles, and contains exactly $1-t''$ long $b$-piles. As per Step 3 of $\strat$, $\mathscr{R}$ places an $r$ chip on the long $b$-pile if $1-t''=1$, on a $b$-singleton if $1-t''=0$ and $k_b-t'>0$, or on an empty pile otherwise. In all three cases, $\mathbb{B}$ is of type I and it is $\mathscr{B}$'s turn to play. Moreover, we have $\mathscr{B}=(0,m_r'')$ and $\mathscr{R}=(0,n_r'')$, where 
$$m_r''\leq m_r-t-t'-t''-1 \qquad \textrm{and} \qquad
n_r''\geq n_r+\sum_{i=1}^{t'+t''+1}|\rho_i|_r.$$
Since $n_r''\geq n_r+\sum_{i=1}^{t'+t''+1}|\rho_i|_r> n_r \geq 0 = m_b$, then, by Theorem~\ref{thm-type1}, $\mathscr{R}$ wins by applying $\strat$ at every round.
\end{proof}

\begin{prop}
\label{prop-typeII-R2}
Let $\mathbb{B}=(k_e,k_r,k_b,1,1)$ be of type II, $\mathscr{B}=\left(m_b, m_r\right)$, $\mathscr{R}=\left(n_b, n_r\right)$, and assume that
$\mathscr{B}$ is the active player. If $m_b\geq 0$ and $m_b+|\beta|_b \leq n_r$ when $\mathscr{B}$'s turn starts, and $\mathscr{B}$ places a chip on $\beta$ before the end of their current round, then $\mathscr{R}$ wins by applying $\strat$ at every round.
\end{prop}

\begin{proof}
Suppose $\mathbb{B}=(k_e,k_r,k_b,1,1)$ is of type II. Assume $\mathscr{B}$ places a chip on $\beta$ before the end of their current round. As such the case where $m_b=m_r=0$ is not possible.

Therefore, $\mathscr{B}$ starts their round by making a sequence of $t+t'+s+1\geq 0$ moves, that consist of placing (refer to Theorem~\ref{thm:SameActive}):
$$
\left\{
    \begin{array}{l}
        \textrm{$t$ $r$ chips on $t$ empty piles,}\\
        \textrm{$t'$ $r$ chips on $t'$ $b$-singletons,}\\
        \textrm{$s$ $b$ chips on $s$ $b$-singletons,}\\
        \textrm{and $1$ chip on $\beta$.}
    \end{array}
\right.
$$
The chip placed on $\beta$ is either (1) a $b$ chip or (2) an $r$ chip.

\begin{enumerate}
\item Suppose $\mathscr{B}$ places a $b$ chip on $\beta$. Then, by the Capture Rule, one chip must be discarded from $\beta$. This chip can either be a guard or a prisoner.
Therefore, we have $\mathscr{B}=(m_b', m_r')$ where $m_b' \leq m_b+|\beta|_b$, $m_r'\leq m_r+|\beta|_r-t-t'$ and $m_b'+m_r'\leq m_b+m_r-t-t'+|\beta| - 1$, and $\mathscr{R}=(n_b',n_r')$ where $n_b' = n_b$ and $n_r'\geq n_r$.
The state of the board $\mathbb{B}$ is now the following:
$$\mathbb{B}=(k_e-t+s+1,k_r+t,k_b-t'-s,t'+1,0). $$
The $t'$ long $r$-piles that were created are all $(b,r)$-piles. Let us denote the $t'+1$ long $r$-piles by $\rho_1, \rho_2, \ldots, \rho_{t'+1}$, respectively. 

Note that $\mathbb{B}$ is now of generalized type I and it is still $\mathscr{B}$'s turn to play.
If $m_b'=m_r'=0$, then $\mathscr{B}$ is eliminated and $\mathscr{R}$ wins.
Otherwise, since
$$m_b' \leq m_b+|\beta|_b \leq n_r \leq n_r' \leq n_r' + \sum_{i=1}^{t'+1}|\rho_i|_r - \max_{1 \leq i \leq t'+1}\big\{\left|\rho_{i}\right|_r\big\} ,$$
then, by Theorem~\ref{thm-gentype1}, $\mathscr{R}$ wins by applying $\strat$ at every round.

\item Suppose $\mathscr{B}$ places an $r$ chip on $\beta$. Then, we have $\mathscr{B}=(m_b', m_r')$ where $m_b'=m_b$ and $m_r'\leq m_r-t-t'-1$, and $\mathscr{R}=(n_b',n_r')$ where $n_b' = n_b$ and $n_r' \geq n_r$. 
The state of the board $\mathbb{B}$ is now the following:
$$\mathbb{B}=(k_e-t+s,k_r+t,k_b-t'-s,t'+2,0). $$
The $t'+1$ long $r$-piles that were created consist of $t'$ $(b,r)$-piles, and $\langle\beta, (r)\rangle$. Let us denote the $t'+2$ long $r$-piles by $\rho_1,\rho_2,\ldots,\rho_{t'+2}$, respectively.

Note that $\mathbb{B}$ is now of generalized type I and it still $\mathscr{B}$'s turn to play. If $m_b'=m_r'=0$, then $\mathscr{B}$ is eliminated and $\mathscr{R}$ wins.
Otherwise, since
$$m_b'= m_b< m_b+|\beta|_b \leq n_r \leq n_r' \leq n_r'+ \sum_{i=1}^{t'+2}|\rho_i|_r- \max_{1 \leq i \leq t'+2} \big\{\left|\rho_{i}\right|_r\big\},$$
then, by Theorem~\ref{thm-gentype1}, $\mathscr{R}$ wins by applying $\strat$ at every round.
\end{enumerate}
\end{proof}

\begin{prop}
\label{prop-typeII-R3}
Let $\mathbb{B}=(k_e,k_r,k_b,1,1)$ be of type II, $\mathscr{B}=\left(m_b, m_r\right)$, $\mathscr{R}=\left(n_b, n_r\right)$, and assume that
$\mathscr{B}$ is the active player. If
\begin{align}
\label{prop-typeII-R3.eq}
m_b\geq 0 \qquad \text{and} \qquad m_b+|\beta|_b \leq n_r
\end{align}
when $\mathscr{B}$'s turn starts, and $\mathscr{B}$ does not place any chip on $\beta$ during their current round, then $\mathscr{R}$ wins by applying $\strat$ at every round.
\end{prop}

We prove Proposition~\ref{prop-typeII-R3} using the following inductive argument on $m_b+m_r$.
\begin{description}
\item[Base case:] We prove that for all game states where $\mathbb{B}=(k_e,k_r,k_b,1,1)$, $\mathscr{B}=\left(0, 0\right)$, $\mathscr{R}=\left(n_b, n_r\right)$, such that $0+|\beta|_b \leq n_r$, if $\mathscr{B}$ is the active player, then $\mathscr{R}$ wins by applying $\strat$.
\item[Inductive step:] Consider a game state where $\mathbb{B}=(k_e,k_r,k_b,1,1)$, $\mathscr{B}=\left(m_b, m_r\right)$, $\mathscr{R}=\left(n_b, n_r\right)$ such that~\eqref{prop-typeII-R3.eq} is satisfied and $m_b+m_r > 0$.
Assume $\mathscr{B}$ is the active player. Moreover, suppose that, no matter what $\mathscr{B}$ plays, as long as $\mathscr{B}$ does not place a chip on $\beta$, $\mathscr{R}$ applies $\strat$ during their round. Then, when $\mathscr{R}$'s round is over, we get $\mathbb{B}=(k_e'',k_r'',k_b'',1,1)$, $\mathscr{B}=\left(m_b'', m_r''\right)$ and  $\mathscr{R}=\left(n_b'', n_r''\right)$, where $\mathscr{B}$ is the active player, \eqref{prop-typeII-R3.eq} is satisfied by $\mathbb{B}$, $\mathscr{B}$ and $\mathscr{R}$, and $m_b''+m_r'' < m_b+m_r$. 

Note that in some cases, instead of using the inductive hypothesis, we will invoke Theorem~\ref{thm-type1}, Proposition~\ref{prop-typeII-R1} or Proposition~\ref{prop-typeII-R2}.
\end{description}

\begin{proof}
Suppose $\mathbb{B}=(k_e,k_r,k_b,1,1)$ is of type II.

Consider a game state where $m_b = m_r = 0$ (base case). Since $\mathscr{B}$ is out of chips, they are eliminated and $\mathscr{R}$ wins.

Consider now a game state where $\mathscr{B}=\left(m_b, m_r\right)$ with $m_b+m_r> 0$, and $\mathscr{R}=\left(n_b, n_r\right)$ such that~\eqref{prop-typeII-R3.eq} is satisfied.
Suppose $\mathscr{B}$ starts their round by making a sequence of $t+t'+s\geq 0$ moves, that consist of placing (refer to Theorem~\ref{thm:SameActive}):
$$
\left\{
	\begin{array}{l}
		\text{$t$ $r$ chips on $t$ empty piles,}\\
		\text{$t'$ $r$ chips on $t'$ $b$-singletons,}\\
		\text{and $s$ $b$ chips on $s$ $b$-singletons.}
	\end{array}
\right.
$$
The state of the board $\mathbb{B}$ is now the following:
$$\mathbb{B}=(k_e-t+s,k_r+t,k_b-t'-s,t'+1,1). $$
Let us denote the $t'+1$ long $r$-piles by $\rho_1,\rho_2,\ldots,\rho_{t'+1}$, respectively.

We have $\mathscr{B}=(m_b',m_r')$, where $m_b'=m_b$ and $m_r'\leq m_r-t-t'$. If $m_b=m_r'=0$, then $\mathscr{B}$ is out of chips. Therefore, $\mathscr{B}$ is eliminated and $\mathscr{R}$ wins. Otherwise, $\mathscr{B}$ can end their round in four different ways (refer to Theorem \ref{thm:DifferentActive}), by: (1) placing a $b$ chip on an empty pile, (2) placing a $b$ chip on an $r$-singleton, (3) placing a $b$ chip on a long $r$-pile, or (4) placing an $r$ chip on an $r$-pile.
Observe that in all cases, $\mathscr{R}$ has at least one $r$ chip since $n_r\geq m_b+|\beta|_b \geq |\beta|_b \geq 1$.
\begin{enumerate}
\item Suppose $\mathscr{B}$ places a $b$ chip on an empty pile. Thus, a $b$-singleton is created and no long pile is created. The move goes to $\mathscr{R}$, who captures all $r$-piles and discards all their prisoners, as per steps 1 and 2 of $\strat$. Then, as per Step 3 of $\strat$, $\mathscr{R}$ places an $r$ chip on $\beta$. Therefore, the board contains only one long pile $\langle \beta,(r) \rangle$. As such, $\mathbb{B}$ is of type I and it is $\mathscr{B}$'s turn to play. Moreover, we have $\mathscr{B}=(m_b'', m_r'')$ and $\mathscr{R}=(0, n_r'')$ where $m_b''=m_b-1$,
$$m_r'' \geq m_r-t-t' \qquad \textrm{and} \qquad n_r'' \geq n_r+\sum_{i=1}^{t'+1}|\rho_i|_r-1.$$
Since $m_b \leq m_b+|\beta|_b \leq n_r \leq n_r+\sum_{i=1}^{t'+1}|\rho_i|_r$ by~\eqref{prop-typeII-R3.eq}, then $m_b-1 \leq n_r+\sum_{i=1}^{t'+1}|\rho_i|_r-1\leq n_r''$. Therefore, by Theorem~\ref{thm-type1}, $\mathscr{R}$ wins by applying $\strat$ at every round.

\item Suppose $\mathscr{B}$ places a $b$ chip on an $r$-singleton. Thus, the long $b$-pile $(r,b)$ is created. The move goes to $\mathscr{R}$, who captures all $r$-piles and discards all their prisoners, as per steps 1 and 2 of $\strat$.
Observe that, since $\beta$ is a long $b$-pile, then $|\beta| = \max\{|\beta|, |(r,b)|\}$.
Then, as per Step 3 of $\strat$, $\mathscr{R}$ places an $r$ chip on $\beta$. Therefore, the board contains exactly two long piles: one long $r$-pile $\langle \beta,(r)\rangle$ and one long $b$-pile $(r,b)$. As such, $\mathbb{B}$ is of type II and it is $\mathscr{B}$'s turn to play. Moreover, we have $\mathscr{B}=(m_b'', m_r'')$ and $\mathscr{R}=(0, n_r'')$ where $m_b'' = m_b-1$,
$$m_r'' \geq m_r-t-t' \qquad \textrm{and} \qquad n_r'' \geq n_r+\sum_{i=1}^{t'+1}|\rho_i|_r-1.$$
We also have
$$m_b'' + |(r,b)|_b = (m_b-1) + |(r,b)|_b \leq (m_b-1)+ |\beta|_b \leq n_r-1 \leq n_r+\sum_{i=1}^{t'+1}|\rho_i|_r - 1 \leq n_r'' $$
by~\eqref{prop-typeII-R3.eq}.
Additionally, $m_b''+m_r''\leq (m_b-1)+(m_r-t-t')<m_b+m_r$. 
\begin{enumerate}
\item If $m_b''= 0$, then by Proposition~\ref{prop-typeII-R1}, $\mathscr{R}$ wins by applying $\strat$ at every round.
\item If $m_b''>0$ and $\mathscr{B}$ places a chip on $\beta$, then by Proposition~\ref{prop-typeII-R2}, $\mathscr{R}$ wins by applying $\strat$ at every round.
\item If $m_b''>0$ and $\mathscr{B}$ does not place a chip on $\beta$, then, by induction, $\mathscr{R}$ wins by applying $\strat$ at every round.
\end{enumerate}

\item Suppose $\mathscr{B}$ places a $b$ chip on a long $r$-pile, say $\rho_1$ without loss of generality. Thus, the long $b$-pile $\langle \rho_1,(b) \rangle$ is created. The move goes to $\mathscr{R}$, who captures all $r$-piles and discards all their prisoners, as per steps 1 and 2 of $\strat$. Let $\pi_{\min},\pi_{\max}\in\{\beta,\langle \rho_1,(b) \rangle\}$ such that $\pi_{\min}\neq\pi_{\max}$, $|\pi_{\min}| =\min\{|\beta|, |\langle \rho_1,(b)\rangle|\}$ and $|\pi_{\max}|=\max\{|\beta|, |\langle \rho_1,(b)\rangle|\}$. Then, as per Step 3 of $\strat$, $\mathscr{R}$ places an $r$ chip on $\pi_{\max}$. Therefore, the board contains exactly two long piles: one long $r$-pile $\langle \pi_{\max},(r) \rangle$ and one long $b$-pile $\pi_{\min}$. As such, $\mathbb{B}$ is of type II and it is $\mathscr{B}$'s turn to play. Moreover, we have $\mathscr{B}=(m_b'', m_r'')$ and $\mathscr{R}=(0, n_r'')$ where $m_b'' = m_b-1$,
$$m_r'' \geq m_r-t-t' \qquad \textrm{and} \qquad n_r'' \geq n_r+\sum_{i=2}^{t'}|\rho_i|_r-1.$$
We also have
$$m_b''+|\pi_{min}|_b = (m_b-1)+|\pi_{min}|_b \leq (m_b-1)+ |\beta|_b \leq n_r-1 \leq n_r+\sum_{i=2}^{t'+1}|\rho_i|_r - 1 \leq n_r'' $$
by the definition of $\pi_{\min}$ and~\eqref{prop-typeII-R3.eq}.
Additionally, $m_b''+m_r''\leq (m_b-1)+(m_r-t-t')<m_b+m_r$. 
The rest of the argument for this case is identical to the one for Case 2.

\item Suppose $\mathscr{B}$ places an $r$ chip on an $r$-pile. This results in a capture for $\mathscr{R}$, from which $\mathscr{R}$ gains at least one $r$ chip. The move goes to $\mathscr{R}$, who captures all $r$-piles and discards all their prisoners, as per steps 1 and 2 of $\strat$. Then, as per Step 3 of $\strat$, $\mathscr{R}$ places an $r$ chip on $\beta$. Therefore, the board contains only one long pile $\langle \beta,(r) \rangle$. As such, $\mathbb{B}$ is of type I and it is $\mathscr{B}$'s turn to play. Moreover, we have $\mathscr{B}=(m_b'', m_r'')$ and $\mathscr{R}=(0, n_r'')$ where $m_b''=m_b$,
$$m_r'' \geq m_r-t-t'-1 \qquad \textrm{and} \qquad n_r'' \geq n_r+\sum_{i=1}^{t'+1}|\rho_i|_r.$$
Since
$$m_b''=m_b \leq m_b+|\beta|_b \leq n_r \leq n_r+\sum_{i=1}^{t'+1}|\rho_i|_r\leq n_r'' $$
by~\eqref{prop-typeII-R3.eq}, then by Theorem~\ref{thm-type1}, $\mathscr{R}$ wins by applying $\strat$ at every round.
\end{enumerate}
\end{proof}

The following corollary is a direct consequence of Propositions~\ref{prop-typeII-R1}, \ref{prop-typeII-R2} and~\ref{prop-typeII-R3}.
\begin{cor}
\label{cor-typeII-R}
Let $\mathbb{B}=(k_e,k_r,k_b,1,1)$ be of type II, $\mathscr{B}=\left(m_b, m_r\right)$, $\mathscr{R}=\left(n_b, n_r\right)$, and assume that $\mathscr{B}$ is the active player. If $m_b= 0$ or $m_b+|\beta|_b \leq n_r$ when $\mathscr{B}$'s turn starts, then $\mathscr{R}$ wins by applying $\strat$ at every round.
\end{cor}

The following theorem is a direct consequence of Proposition~\ref{prop-typeII-B} and Corollary~\ref{cor-typeII-R}.
\begin{theorem}\label{thm-type2}
Let $\mathbb{B}=(k_e,k_r,k_b,1,1)$ be of type II, $\mathscr{B}=\left(m_b, m_r\right)$, $\mathscr{R}=\left(n_b, n_r\right)$, and assume that $\mathscr{B}$ is the active player.
Then, $\mathscr{B}$ has a winning strategy if and only if $m_b> 0$ and $m_b+|\beta|_b>n_r$ when $\mathscr{B}$'s turn starts. 
Whenever there is a winning strategy for $\mathscr{B}$ (respectively for $\mathscr{R}$), then $\strat$ is such a strategy.
\end{theorem}

\subsection{Generalized Type II Board}
\label{subsec:gentypeII}

We slightly extend the definition of type II boards in the following way.
\begin{definition}
The board $\mathbb{B}$ is said to be of \emph{generalized type II} if $\mathbb{B}=(k_e,k_r,k_b,1,h)$ with $h\geq 1$ We denote the long $r$-pile by $\rho$, and the $h$ long $b$-piles by $\beta_1, \ldots, \beta_h$.
\end{definition}


The structure of this subsection is the same as the one for Subsection~\ref{subsec:typeII}. We first describe the states where there is a winning strategy for $\mathscr{B}$ on a generalized type II board (refer to Proposition~\ref{prop-gentype2-B}). Then we describe the states where there is a winning strategy for $\mathscr{R}$ on a generalized type II board(refer to Propositions~\ref{prop-gentype2-R1} and~\ref{prop-gentype2-R2}). We finally present the characterization of the winning states for the first player in Theorem~\ref{thm-gentype2}.


\begin{prop}
\label{prop-gentype2-B}
Let $\mathbb{B}=(k_e,k_r,k_b,1,h)$ be of generalized type II, $\mathscr{B}=\left(m_b, m_r\right)$, $\mathscr{R}=\left(n_b, n_r\right)$, and assume that $\mathscr{B}$ is the active player. If
\begin{align}
\label{prop-gentype2-B.eq}
m_b>0 \qquad \text{and} \qquad m_b+\sum_{i=1}^{h}\left|\beta_{i}\right|_b>n_r
\end{align}
when $\mathscr{B}$'s turn starts, then $\mathscr{B}$ wins by applying $\strat$ at every round.
\end{prop}

\begin{proof}
Suppose $\mathbb{B}=(k_e,k_r,k_b,1,h)$ is of generalized type II. Consider a game state where $\mathscr{B}=(m_b,m_r)$ and $\mathscr{R}=(n_b,n_r)$, with $m_b>0$ and $m_b+\sum_{i=1}^{h}|\beta_i|_b>n_r$. Since $m_b>0$, then $\mathscr{B}$ has at least one $b$ chip. Player $\mathscr{B}$ applies Strategy $\strat$: they capture all $b$-piles and discard all their prisoners, as per steps 1 and 2 of $\strat$. The state of the board $\mathbb{B}$ is now the following:
$$\mathbb{B}=(k_e+k_b+h,k_r,0,1,0).$$

We now have a board of type I, with active player $\mathscr{B}=(m_b', 0)$ where $m_b'\geq m_b+\sum_{i=1}^{h}|\beta_i|_b$. Moreover, $\mathscr{R}=(n_b', n_r')$ where $n_b'\leq n_b$ and $n_r' = n_r$. 

By~\eqref{prop-gentype2-B.eq}, we have $m_b'\geq m_b+\sum_{i=1}^{h}|\beta_i|_b > n_r = n_r'$. Therefore, by Theorem~\ref{thm-type1}, $\mathscr{B}$ wins by applying $\strat$ at every round.
\end{proof}

\begin{prop}
\label{prop-gentype2-R1}
Let $\mathbb{B}=(k_e,k_r,k_b,1,h)$ be of generalized type II, $\mathscr{B}=\left(m_b, m_r\right)$, $\mathscr{R}=\left(n_b, n_r\right)$, and assume that $\mathscr{B}$ is the active player. If $m_b=0$ when $\mathscr{B}$'s turn starts, then $\mathscr{R}$ wins by applying $\strat$ at every round.
\end{prop}

We prove Proposition~\ref{prop-gentype2-R1} using the following inductive argument on $h$.
\begin{description}
\item[Base case:] We prove that for all game states where $\mathbb{B}=(k_e,k_r,k_b,1,1)$, $\mathscr{B}=\left(0, m_r\right)$, $\mathscr{R}=\left(n_b, n_r\right)$, if $\mathscr{B}$ is the active player, then $\mathscr{R}$ wins by applying $\strat$.
\item[Inductive step:] Consider a game state where $\mathbb{B}=(k_e,k_r,k_b,1,h)$, $\mathscr{B}=\left(0, m_r\right)$, and $\mathscr{R}=\left(n_b, n_r\right)$. Assume $\mathscr{B}$ is the active player. Moreover, suppose that no matter what $\mathscr{B}$ plays, $\mathscr{R}$ applies $\strat$ during their round. Then, when $\mathscr{R}$'s round is over, we get $\mathbb{B}=(k_e'',k_r'',k_b'',1,h'')$, $\mathscr{B}=\left(0, m_r''\right)$ and  $\mathscr{R}=\left(n_b'', n_r''\right)$, where $\mathscr{B}$ is the active player, and $h'' < h$.

Note that in some cases, instead of using the inductive hypothesis, we will invoke Theorem~\ref{thm-type1}.
\end{description}

\begin{proof}
Suppose $\mathbb{B}=(k_e,k_r,k_b,1,h)$.

Consider a game state where $h=1$ (base case). Then $\mathbb{B}$ is of type II. By Theorem~\ref{thm-type2}, $\mathscr{R}$ wins by applying $\strat$.

Consider a game state where $\mathscr{B}=(0,m_r)$, $\mathscr{R}=(n_b,n_r)$, and $h> 1$. If $m_r=0$, then $\mathscr{B}$ is out of chips, they are eliminated and $\mathscr{R}$ wins. Otherwise, $\mathscr{B}$ starts their round by making a sequence of $t+t'+t''\geq 0$ moves, that consist of placing (refer to Theorem~\ref{thm:SameActive}):
$$
\left\{
	\begin{array}{l}
		\text{$t$ $r$ chips on $t$ empty piles,}\\
		\text{$t'$ $r$ chips on $t'$ $b$-singletons,}\\
		\text{and $t''$ $r$ chips on $t''$ long $b$-piles.}
\end{array}\right.
$$
The state of the board $\mathbb{B}$ is now the following:
$$\mathbb{B}=(k_e-t,k_r+t,k_b-t',t'+t''+1,h-t''). $$

We have $\mathscr{B}=(0,m_r')$ where $m_r'\leq m_r-t-t'-t''$. If $m_r'=0$, then $\mathscr{B}$ is out chips, they are eliminated and $\mathscr{R}$ wins. Otherwise, $\mathscr{B}$ must end their round by placing an $r$ chip on an $r$-pile (refer to Theorem~\ref{thm:DifferentActive}). This results in a capture for $\mathscr{R}$, from which $\mathscr{R}$ gains at least one $r$ chip. The move goes to $\mathscr{R}$, who captures all $r$-piles and discards all their prisoners, as per steps 1 and 2 of $\strat$. Now, the only long piles on the board are the $h-t''$ long $b$-piles. There are two cases to consider: (1) if $h-t'' = 0$, (2) or if $h-t''>0$.
\begin{enumerate}
\item If $h-t''=0$, then $\mathscr{R}$ places an $r$ chip on a $b$-singleton if there is one, or on an empty pile otherwise. In both cases, $\mathbb{B}$ is of type I and it is $\mathscr{B}$'s turn to play. Since $m_b=0$, then by Theorem \ref{thm-type1}, $\mathscr{B}$ wins by applying $\strat$ at every round.
\item If $h-t''>0$, then $\mathscr{R}$ places an $r$ chip on the $b$-pile of greatest length. As such, $\mathbb{B}$ is of generalized type II with $h-t''-1<h$ long $b$-piles and it is $\mathscr{B}$'s turn to play.
\end{enumerate}
\end{proof}

\begin{prop}
\label{prop-gentype2-R2}
Let $\mathbb{B}=(k_e,k_r,k_b,1,h)$ be of generalized type II board, $\mathscr{B}=\left(m_b, m_r\right)$, $\mathscr{R}=\left(n_b, n_r\right)$, and assume that $\mathscr{B}$ is the active player. If
\begin{align}
\label{prop-gentype2-R2.eq}
m_b>0 \qquad \textrm{and} \qquad m_b+\sum_{i=1}^{h}|\beta_i|_b \leq n_r
\end{align}
when $\mathscr{B}$'s turn starts, then $\mathscr{R}$ wins by applying $\strat$ at every round.
\end{prop}

Let
\begin{align}
\label{eq.def.mu}
\mu((\beta_1,\ldots,\beta_h),m_b,m_r)=m_b+m_r+\sum_{i=1}^{h}(|\beta_i|-1)-1.
\end{align}
Intuitively, $\mu((\beta_1,\ldots,\beta_h),m_b,m_r)$ corresponds to the total number of chips in $\mathscr{B}$'s possession if, during their current round, $\mathscr{B}$ were to capture all $b$-piles on the board, keep all prisoners they are allowed to keep, and end their round by placing a chip on the board. 
We prove Proposition~\ref{prop-gentype2-R2} using the following inductive argument on $\mu((\beta_1,\ldots,\beta_h),m_b,m_r)$.
\begin{description}
\item[Base case:] We prove that for all game states where $\mathbb{B}=(k_e,k_r,k_b,1,h)$, $\mathscr{B}=\left(m_b, m_r\right)$, $\mathscr{R}=\left(n_b, n_r\right)$,~\eqref{prop-gentype2-R2.eq} is satisfied and $\mu((\beta_1,\ldots,\beta_h),m_b,m_r) = 1$, if $\mathscr{B}$ is the active player, then $\mathscr{R}$ wins by applying $\strat$.
\item[Inductive step:] Consider a game state where $\mathbb{B}=(k_e,k_b,k_r,1,h)$, $\mathscr{B}=\left(m_b, m_r\right)$, $\mathscr{R}=\left(n_b, n_r\right)$,~\eqref{prop-gentype2-R2.eq} is satisfied, and $\mu((\beta_1,\ldots,\beta_h),m_b,m_r) > 1$. Assume $\mathscr{B}$ is the active player. Moreover, suppose that, no matter what $\mathscr{B}$ plays, $\mathscr{R}$ applies $\strat$ during their round. Then, when $\mathscr{R}$'s round is over, we get $\mathbb{B}=(k_e'',k_r'',k_b'',1,h'')$, $\mathscr{B}=\left(m_b'', m_r''\right)$ and $\mathscr{R}=\left(n_b'', n_r''\right)$, where $\mathscr{B}$ is the active player,~\eqref{prop-gentype2-R2.eq} is satisfied by $m_b''$, $n_r''$ and $(\beta_1'',\ldots,\beta_h'')$, and $\mu((\beta_1'',\ldots,\beta_h''),m_b'',m_r'') < \mu((\beta_1,\ldots,\beta_h),m_b,m_r)$.

Note that in some cases, instead of using the inductive hypothesis, we will invoke Theorem~\ref{thm-type1} or Proposition~\ref{prop-gentype2-R1}.
\end{description}

\begin{proof}
Suppose $\mathbb{B}=(k_e,k_r,k_b,1,h)$ is of generalized type II, and consider the parameter $\mu((\beta_1,\ldots,\beta_h),m_b,m_r)$ (refer to~\eqref{eq.def.mu}). We need the following observation for the base case. Since $m_b \geq 1$, $m_r\geq 0$, $h\geq 1$ and $\beta_1$ is a long pile, we have
\begin{align*}
\mu((\beta_1,\ldots,\beta_h),m_b,m_r) &= m_b+m_r+\sum_{i=1}^{h}(|\beta_i|-1)-1 \\
&\geq 1+0+\sum_{i=1}^{1}(|\beta_i|-1)-1 \\
&= |\beta_1|-1 \\
&\geq 1 ,
\end{align*}
with equality whenever $m_b = 1$, $m_r = 0$, $h = 1$ and $\beta_1 = (r,b)$. Otherwise, $\mu((\beta_1,\ldots,\beta_h),m_b,m_r) \geq 2$.

Consider a game state where $\mu((\beta_1,\ldots,\beta_h),m_b,m_r) = 1$ (base case). Then, by the previous observation, $\mathbb{B}$ is of type II. Moreover, \eqref{prop-gentype2-R2.eq} becomes $m_b + |\beta_1|_b \leq n_r$. Thus, by Theorem~\ref{thm-type2}, $\mathscr{R}$ wins by applying $\strat$ at every round. 

Consider now a game state where $\mathscr{B}=\left(m_b, m_r\right)$, $\mathscr{R}=\left(n_b, n_r\right)$ and~\eqref{prop-gentype2-R2.eq} is satisfied. Moreover, assume $\mu((\beta_1,\ldots,\beta_h),m_b,m_r) > 1$. Since $m_b > 0$, $\mathscr{B}$ has at least one $b$ chip. Suppose $\mathscr{B}$ starts their round by making a sequence of $t+t'+t''+s+s'\geq 0$ moves, that consist of placing (refer to Theorem \ref{thm:SameActive}):
$$
\left\{
    \begin{array}{l}
		\text{$t$ $r$ chips on $t$ empty piles,}\\
	  \text{$t'$ $r$ chips on $t'$ $b$-singletons,}\\
		\text{$t''$ $r$ chips on $t''$ long $b$-piles,}\\
		\text{$s$ $b$ chips on $s$ $b$-singletons,}\\
		\text{and $s'$ $b$ chips on $s'$ long $b$-piles.}
	\end{array}
\right.
$$
The state of the board $\mathbb{B}$ is now the following:
$$\mathbb{B}=(k_e-t+s+s',k_r+t,k_b-t'-s,t'+t''+1,h-t''-s). $$
Let us denote the $s'$ long $b$-piles that were captured by $\beta_1,\ldots,\beta_{s'}$ and the remaining long $b$-piles on the board by $\beta_{s'+1},\ldots, \beta_{h-t''}$. The $t''$ long $r$-piles that were created from long $b$-piles are $\langle\beta_{h-t''+1},(r)\rangle,\langle\beta_{h-t''+2},(r)\rangle,\ldots,\langle\beta_{h},(r)\rangle$. We rename the $t'+t''+1$ long $r$-piles as $\rho_1,\ldots,\rho_{t'+t''+1}$.

For each pile captured by $\mathscr{B}$, one chip must be discarded (by the Capture Rule). This chip can either be a guard or a prisoner. Therefore, we have $\mathscr{B}=(m_b',m_r')$ where 
\begin{align*}
m_b' &\leq m_b+\sum_{i=1}^{s'}|\beta_i|_b,\\
m_r' &\leq m_r-t-t'-t''+\sum_{i=1}^{s'}|\beta_i|_r,\\
m_b'+m_r' &\leq m_b+m_r-t-t'-t''+\sum_{i=1}^{s'}(|\beta_i|-1).
\end{align*}

Player $\mathscr{B}$ can end their round in four different ways (refer to Theorem \ref{thm:DifferentActive}): (1) placing a $b$ chip on an empty pile, (2) placing a $b$ chip on an $r$-singleton, (3) placing a $b$ chip on a long $r$-pile, or (4) placing an $r$ chip on an $r$-pile. Note that \eqref{prop-gentype2-R2.eq} implies that $n_r\geq 1$, i.e., $\mathscr{R}$ has at least one $r$ chip and is able to play once $\mathbb{B}$ finishes their first round.
\begin{enumerate}
\item Suppose $\mathscr{B}$ places a $b$ chip on an empty pile. Then, a $b$-singleton is created and no long pile is created. 
The move goes to $\mathscr{R}$, who captures all $r$-piles and discards all their prisoners, as per steps 1 and 2 of $\strat$. 
Therefore, there are $h-t''-s'$ long piles, which are all long $b$-piles.
Then, as per Step 3 of $\strat$, $\mathscr{R}$ places an $r$ chip on the $b$-pile of greatest length. Thus, it is $\mathscr{B}$'s turn to play, and we have
$\mathscr{B}=(m_b'',m_r'')$ and $\mathscr{R}=(0,n_r'')$, where
\begin{align}
\label{ineq.mbpp.4.11}
m_b'' &\leq m_b+\sum_{i=1}^{s'}|\beta_i|_b-1, \\
\label{ineq.mrpp.4.11}
m_r'' &\leq m_r-t-t'-t''+\sum_{i=1}^{s'}|\beta_i|_r, \\
\label{ineq.summbmrpp.4.11}
m_b''+m_r'' &\leq m_b+m_r-t-t'-t'' +\sum_{i=1}^{s'}(|\beta_i|-1)-1, \\
\label{ineq.nrpp.4.11}
n_r'' &\geq n_r+\sum_{i=1}^{t'+t''+1}\left|\rho_{i}\right|_r-1.
\end{align}
The state of the board $\mathbb{B}$ depends on the value of $h-t''-s'$: either (a) $0\leq h-t''-s' \leq 1$ or (b) $h-t''-s' > 1$.
\begin{enumerate}
\item Suppose $0\leq h-t''-s' \leq 1$. If $h-t''-s' = 0$, then $\mathscr{R}$ ended their round by placing an $r$ chip on a $b$-singleton. If $h-t''-s' = 1$, then $\mathscr{R}$ ended their round by placing an $r$ chip on the only long $b$-pile that was on the board. In both cases, $\mathbb{B}$ now contains only one long pile, which is a long $r$-pile. As such, $\mathbb{B}$ is of type I and it is $\mathscr{B}$'s turn to play. Observe that $m_b''<n_r''$. Indeed, we have
\begin{align*}
m_b'' \leq\: & m_b+\sum_{i=1}^{s'}|\beta_i|_b-1 & \text{by~\eqref{ineq.mbpp.4.11},} \\
\leq\: &m_b+\sum_{i=1}^{h-t''}|\beta_i|_b -1 &\text{since $0\leq h-t''-s'$,}\\
<\: &m_b+\sum_{i=1}^{h}|\beta_i|_b \\
\leq\: &n_r  &\text{by~\eqref{prop-gentype2-R2.eq},}\\
\leq\: &n_r+\sum_{i=1}^{t'+t''+1}|\rho_i|_r-1\\
\leq\: & n_r'' & \text{by~\eqref{ineq.nrpp.4.11}.}
\end{align*}
Thus, by Theorem~\ref{thm-type1}, $\mathscr{R}$ wins by applying $\strat$ at every round.

\item Suppose $h-t''-s' > 1$. Then $\mathscr{R}$ ended their round by placing an $r$ chip on the $b$-pile of greatest length, say $\beta_{h-t''}$ without loss of generality. Thus, $\mathbb{B}$ contains one long $r$-pile $\langle\beta_{h-t''},(r)\rangle$ and $h-t''-s'-1 > 0$ long $b$-piles. As such, $\mathbb{B}$ is of generalized type II and it is $\mathscr{B}$'s turn to play. If $m_b'' = 0$, then by Proposition~\ref{prop-gentype2-R1}, $\mathscr{R}$ wins by applying $\strat$ at every round. Otherwise, $m_b'' > 0$ and the first part of~\eqref{prop-gentype2-R2.eq} is satisfied by $m_b''$ and $n_r''$. Moreover, we have
\begin{align*}
m_b'' + \sum_{i=s'+1}^{h-t''-1}|\beta_i|_b \leq\: & m_b+\sum_{i=1}^{s'}|\beta_i|_b-1 + \sum_{i=s'+1}^{h-t''-1}|\beta_i|_b & \text{by~\eqref{ineq.mbpp.4.11},} \\
=\: & m_b + \sum_{i=1}^{h-t''-1}|\beta_i|_b-1\\
<\: & m_b + \sum_{i=1}^{h}|\beta_i|_b-1 \\
\leq\: & n_r -1 &\text{by~\eqref{prop-gentype2-R2.eq},}\\
\leq\: &n_r+\sum_{i=1}^{t'+t''+1}\left|\rho_{i}\right|_r-1\\
\leq\: &n_r'' & \text{by~\eqref{ineq.nrpp.4.11}.}
\end{align*}
Therefore, \eqref{prop-gentype2-R2.eq} is satisfied by $m_b''$, $m_r''$ and $(\beta_{s+1},\ldots,\beta_{h-t''-1})$. Finally, we have
\begin{align*}
\: &
\mu((\beta_{s+1},\ldots,\beta_{h-t''-1}),m_b'',m_r'') \\
=\: &m_b''+m_r''+\sum_{i=s'+1}^{h-t''-1}(|\beta_i|-1)-1 \\
\leq\: & m_b+m_r-t-t'-t'' +\sum_{i=1}^{s'}(|\beta_i|-1)-1\sum_{i=s'+1}^{h-t''-1}(|\beta_i|-1)-1 & \text{by~\eqref{ineq.summbmrpp.4.11},}\\
\leq\: & m_b+m_r+\sum_{i=1}^{h-t''-1}(|\beta_i|-1)-2 \\
<\: & m_b+m_r+\sum_{i=1}^{h}(|\beta_i|-1)-1 \\
=\: & \mu((\beta_1,\ldots,\beta_h),m_b,m_r).
\end{align*}
\end{enumerate}

\item Suppose $\mathscr{B}$ places a $b$ chip on an $r$-singleton. Then, one long $b$-pile is created, namely $(r,b)$.
The move goes to $\mathscr{R}$, who captures all $r$-piles and discards all prisoners, as per steps 1 and 2 of $\strat$.
Therefore, there are $h-t''-s'+1$ long piles on the board, which are all long $b$-piles.
Then, as per Step 3 of $\strat$, $\mathscr{R}$ places an $r$ chip on the $b$-pile of greatest length. Thus, it is $\mathscr{B}$'s turn to play, and we have
$\mathscr{B}=(m_b'',m_r'')$ and $\mathscr{R}=(0,n_r'')$, where~\eqref{ineq.mbpp.4.11}, \eqref{ineq.mrpp.4.11}, \eqref{ineq.summbmrpp.4.11} and~\eqref{ineq.nrpp.4.11} are satisfied.
The state of the board $\mathbb{B}$ depends on the value of $h-t''-s'+1$: either (a) $h-t''-s' +1 = 1$ or (b) $h-t''-s' +1 > 1$.
\begin{enumerate}
\item Suppose $h-t''-s' +1 = 1$. Then $\mathscr{R}$ ended their round by placing an $r$ chip on the only long $b$-pile that was on the board. As such, the board now contains only one long pile, which is a long $r$-pile. As such, $\mathbb{B}$ is of type I and it is $\mathscr{B}$'s turn to play. 
Observe that $m_b''<n_r''$. The derivation for this inequality is the same as in Case 1(a).
Thus, by Theorem~\ref{thm-type1}, $\mathscr{R}$ wins by applying $\strat$ at every round.

\item Suppose $h-t''-s' +1 > 1$. Then $\mathscr{R}$ ended their round by placing an $r$ chip on the $b$-pile of greatest length, say $\beta_{h-t''}$ without loss of generality. Thus, the board contains one long $r$-pile $\langle\beta_{h-t''},(r)\rangle$ and $h-t''-s' > 0$ long $b$-piles. As such, $\mathbb{B}$ is of generalized type II and it is $\mathscr{B}$'s turn to play. If $m_b'' = 0$, then by Proposition~\ref{prop-gentype2-R1}, $\mathscr{R}$ wins by applying $\strat$ at every round. Otherwise, $m_b'' > 0$ and the first part of~\eqref{prop-gentype2-R2.eq} is satisfied by $m_b''$ and $n_r''$.
Moreover, we have
\begin{align*}
m_b'' + \sum_{i=s'+1}^{h-t''-1}|\beta_i|_b + |(r,b)|_b \leq\: & m_b+\sum_{i=1}^{s'}|\beta_i|_b-1 + \sum_{i=s'+1}^{h-t''-1}|\beta_i|_b + |(r,b)|_b & \text{by~\eqref{ineq.mbpp.4.11},}\\
=\: & m_b + \sum_{i=1}^{h-t''-1}|\beta_i|_b+|(r,b)|_b-1\\
\leq\: & m_b + \sum_{i=1}^{h}|\beta_i|_b -1 &\text{since $|\beta_h|_b\geq |(r,b)|_b$,}\\
\leq\: & n_r -1 &\text{by~\eqref{prop-gentype2-R2.eq},}\\
\leq\: &n_r+\sum_{i=1}^{t'+t''+1}\left|\rho_{i}\right|_r -1\\
\leq\: &n_r'' & \text{by~\eqref{ineq.nrpp.4.11}.}
\end{align*}
Therefore, \eqref{prop-gentype2-R2.eq} is satisfied by $m_b''$, $m_r''$ and $(\beta_{s'+1},\ldots,\beta_{h-t''-1})$. Finally, we have
\begin{align*}
\: & \mu((\beta_{s'+1},\ldots,\beta_{h-t''-1}),m_b'',m_r'') \\
=\: &m_b''+m_r''+\sum_{i=s'+1}^{h-t''-1}(|\beta_i|-1) + (|(r,b)|-1) -1\\
\leq\; & m_b+m_r-t-t'-t'' +\sum_{i=1}^{s'}(|\beta_i|-1)-1 + \sum_{i=s'+1}^{h-t''-1}(|\beta_i|-1) + (|(r,b)|-1) -1 & \text{by~\eqref{ineq.summbmrpp.4.11},}\\
\leq\: & m_b+m_r+\sum_{i=1}^{h-t''-1}(|\beta_i|-1)+ (|(r,b)|-1)-2 \\
<\: & m_b+m_r+\sum_{i=1}^{h}(|\beta_i|-1)-1 &\text{since $|\beta_h| \geq |(r,b)|$,}\\
=\: & \mu((\beta_1,\ldots,\beta_h),m_b,m_r).
\end{align*}
\end{enumerate}

\item Suppose $\mathscr{B}$ places a $b$ chip on a long $r$-pile. 
Then, one long $b$-pile is created. Without loss of generality, let it be $\langle\rho_{t'+t''+1},(b)\rangle$.
The move goes to $\mathscr{R}$, who captures all $r$-piles and discards all their prisoners, as per steps 1 and 2 of $\strat$.
Therefore, there are $h-t''-s'+1$ long piles, which are all long $b$-piles.
Then, as per Step 3 of $\strat$, $\mathscr{R}$ places an $r$ chip on the $b$-pile of greatest length. Thus, it is $\mathscr{B}$'s turn to play, and we have $\mathscr{B}=(m_b'',m_r'')$ and $\mathscr{R}=(0,n_r'')$ where
\begin{align} 
\label{ineq.mbpp.4.11.case.3}
m_b'' &\leq m_b+\sum_{i=1}^{s'}|\beta_i|_b-1,\\
\label{ineq.mrpp.4.11.case.3}
m_r''&\leq m_r-t-t'-t''+\sum_{i=1}^{s'} |\beta_i|_r,\\
\label{ineq.summbmrpp.4.11.case.3}
m_b''+m_r''&\leq m_b+m_r-t-t'-t''+\sum_{i=1}^{s'}(|\beta_i|-1)-1,\\
\label{ineq.nrpp.4.11.case.3} 
n_r''&\geq n_r+\sum_{i=1}^{t'+t''}\left|\rho_{i}\right|_r-1.
\end{align}
The state of the board $\mathbb{B}$ depends on the value of $h-t''-s'+1$: either (a) $h-t''-s' +1 = 1$ or (b) $h-t''-s' +1 > 1$.
\begin{enumerate}
\item Suppose $h-t''-s' +1 = 1$. Then $\mathscr{R}$ ended their round by placing an $r$ chip on the only long $b$-pile that was on the board. Therefore, the board now contains only one long pile, which is a long $r$-pile. As such, $\mathbb{B}$ is of type I and it is $\mathscr{B}$'s turn to play. Observe that $m_b''\leq n_r''$. Indeed, we have
\begin{align*}
m_b'' \leq\: & m_b+\sum_{i=1}^{s'}|\beta_i|_b-1 & \text{by~\eqref{ineq.mbpp.4.11.case.3},}\\
=\: &m_b+\sum_{i=1}^{h-t''}|\beta_i|_b -1 &\text{since $h-t''-s' = 0$,}\\
\leq\: &m_b+\sum_{i=1}^{h}|\beta_i|_b -1\\
\leq\: &n_r -1 &\text{by~\eqref{prop-gentype2-R2.eq},}\\
\leq\: &n_r+\sum_{i=1}^{t'+t''}|\rho_i|_r-1\\
\leq\: & n_r'' & \text{by~\eqref{ineq.nrpp.4.11.case.3}.}
\end{align*}
Thus, by Theorem~\ref{thm-type1}, $\mathscr{R}$ wins by applying $\strat$ at every round.

\item Suppose $h-t''-s' +1 > 1$. Then $\mathscr{R}$ ended their round by placing an $r$ chip on the $b$-pile of greatest length, denoted by $\pi_{\max}$. Observe that 
$$\pi_{\max} \in \{\langle \rho_{t'+t''+1},(b)\rangle\} \cup\{\beta_i\mid s'+1\leq i \leq h-t'' \}.$$
Thus, the board contains one long $r$-pile $\langle\pi_{\max},(r)\rangle$ and $h-t''-s' > 0$ long $b$-piles. As such, $\mathbb{B}$ is of generalized type II and it is $\mathscr{B}$'s turn to play. If $m_b'' = 0$, then by Proposition~\ref{prop-gentype2-R1}, $\mathscr{R}$ wins by applying $\strat$ at every round. Otherwise, $m_b'' > 0$ and the first part of~\eqref{prop-gentype2-R2.eq} is satisfied by $m_b''$ and $n_r''$. Moreover, we have
\begin{align*}
\: & m_b'' + \sum_{i=s'+1}^{h-t''}|\beta_i|_b + |\langle\rho_{t'+t''+1},(b)\rangle|_b - |\pi_{\max}|_b\\
\leq\: & m_b+\sum_{i=1}^{s'}|\beta_i|_b -1 + \sum_{i=s'+1}^{h-t''}|\beta_i|_b + |\langle\rho_{t'+t''+1},(b)\rangle|_b  - |\pi_{\max}|_b & \text{by~\eqref{ineq.mbpp.4.11.case.3},}\\
=\: & m_b+\sum_{i=1}^{h-t''}|\beta_i|_b + |\langle\rho_{t'+t''+1},(b)\rangle|_b  - |\pi_{\max}|_b -1\\
\leq\: & m_b + \sum_{i=1}^{h}|\beta_i|_b -1 & \text{since $|\langle\rho_{t'+t''+1},(b)\rangle|_b  \leq |\pi_{\max}|_b$,}\\
\leq\: & n_r -1 &\text{by~\eqref{prop-gentype2-R2.eq},}\\
\leq\: &n_r+\sum_{i=1}^{t'+t''}\left|\rho_{i}\right|_r -1\\
\leq\: &n_r'' & \text{by~\eqref{ineq.nrpp.4.11.case.3}.}
\end{align*}
Therefore, \eqref{prop-gentype2-R2.eq} is satisfied by $m_b''$, $m_r''$ and $(\beta_{s'+1},\ldots,\beta_{h-t''})$. Finally, we have
\begin{align*}
\: & \mu((\beta_{s'+1},\ldots,\beta_{h-t''}),m_b'',m_r'') \\
=\: & m_b''+m_r'' + \sum_{i=s'+1}^{h-t''}(|\beta_i|-1) + (|\langle\rho_{t'+t''+1},(b)\rangle|-1) - (|\pi_{\max}|-1)  - 1\\
\leq\: &m_b+m_r-t-t'-t''+\sum_{i=1}^{s'}(|\beta_i|-1) -1 \\
& + \sum_{i=s'+1}^{h-t''}(|\beta_i|-1) + (|\langle\rho_{t'+t''+1},(b)\rangle|-1) - (|\pi_{\max}|-1)  - 1 \qquad\qquad\qquad\qquad \text{by~\eqref{ineq.mbpp.4.11.case.3},}\\
\leq\: & m_b+m_r+\sum_{i=1}^{h-t''}(|\beta_i|-1)  + (|\langle\rho_{t'+t''+1},(b)\rangle|-1) - (|\pi_{\max}|-1)  - 2\\
<\: & m_b+m_r+\sum_{i=1}^{h}(|\beta_i| -1) - 1 \qquad\qquad\qquad\qquad\qquad\qquad\text{since $|\langle\rho_{t'+t''+1},(b)\rangle| \leq |\pi_{\max}|$,}\\
=\: & \mu((\beta_1,\ldots,\beta_h),m_b,m_r).
\end{align*}
\end{enumerate}

\item Suppose $\mathscr{B}$ places an $r$ chip on an $r$-pile. This results in a capture for $\mathscr{R}$, from which $\mathscr{R}$ gains at least one $r$ chip. The move goes to $\mathscr{R}$, who captures all $r$-piles and discards all their prisoners, as per steps 1 and 2 of $\strat$. Therefore, there are $h-t''-s'$ long piles, which are all long $b$-piles. Then, as per Step 3 of $\strat$, $\mathscr{R}$ places an $r$ chip on the $b$-pile of greatest length. Thus, it is $\mathscr{B}$'s turn to play, and we have $\mathscr{B}=(m_b'',m_r'')$ and $\mathscr{R}=(0,n_r'')$, where
\begin{align} 
\label{ineq.mbpp.4.11.case.4}
m_b'' &\leq m_b+\sum_{i=1}^{s'}|\beta_i|_b,\\
\label{ineq.mrpp.4.11.case.4}
m_r'' &\leq m_r-t-t'-t''+\sum_{i=1}^{s'}|\beta_i|_r-1,\\
\label{ineq.summbmrpp.4.11.case.4}
m_b''+m_r'' &\leq m_b+m_r-t-t'-t'' +\sum_{i=1}^{s'}(|\beta_i|-1)-1,\\
\label{ineq.nrpp.4.11.case.4}
n_r'' &\geq n_r+\sum_{i=1}^{t'+t''+1}\left|\rho_{i}\right|_r.
\end{align}
The state of the board $\mathbb{B}$ depends on the value of $h-t''-s'$: either (a) $0\leq h-t''-s' \leq 1$ or (b) $h-t''-s' > 1$.
\begin{enumerate}
\item Suppose $0\leq h-t''-s' \leq 1$. If $h-t''-s' = 0$, then $\mathscr{R}$ ended their round by placing an $r$ chip either on a $b$-singleton if $k_b-t'-s>0$ or on an empty pile if $k_b-t'-s=0$. If $h-t''-s' = 1$, then $\mathscr{R}$ ended their round by placing an $r$ chip on the only long $b$-pile that was on the board. In all cases, the board now contains at most one long pile, which is a long $r$-pile if it exists. As such, $\mathbb{B}$ is of type I and it is $\mathscr{B}$’s turn to play. Observe that $m_b''< n_r''$. Indeed, we have
\begin{align*}
m_b'' \leq\: & m_b+\sum_{i=1}^{s'}|\beta_i|_b & \text{by~\eqref{ineq.mbpp.4.11.case.4},}\\
\leq\: &m_b+\sum_{i=1}^{h-t''}|\beta_i|_b &\text{since $0\leq h-t''-s'$,}\\
\leq\: &m_b+\sum_{i=1}^{h}|\beta_i|_b \\
\leq\: &n_r  &\text{by~\eqref{prop-gentype2-R2.eq},}\\
<\: &n_r+\sum_{i=1}^{t'+t''+1}|\rho_i|_r\\
\leq\: & n_r'' & \text{by~\eqref{ineq.nrpp.4.11.case.4}.}
\end{align*}
Thus, by Theorem~\ref{thm-type1}, $\mathscr{R}$ wins by applying $\strat$ at every round.

\item Suppose $h-t''-s' > 1$. Then $\mathscr{R}$ ended their round by placing an $r$ chip on the $b$-pile of greatest length, say $\beta_{h-t''}$ without loss of generality. Thus, the board contains one long $r$-pile $\langle\beta_{h-t''},(r)\rangle$ and $h-t''-s'-1 > 0$ long $b$-piles. As such, $\mathbb{B}$ is of generalized type II and it is $\mathscr{B}$'s turn to play. If $m_b'' = 0$, then by Proposition~\ref{prop-gentype2-R1}, $\mathscr{R}$ wins by applying $\strat$ at every round. Otherwise, $m_b'' > 0$ and the first part of~\eqref{prop-gentype2-R2.eq} is satisfied by $m_b''$ and $n_r''$. Moreover, we have
\begin{align*}
m_b'' + \sum_{i=s'+1}^{h-t''-1}|\beta_i|_b \leq\: & m_b+\sum_{i=1}^{s'}|\beta_i|_b + \sum_{i=s'+1}^{h-t''-1}|\beta_i|_b  & \text{by~\eqref{ineq.mbpp.4.11.case.4},}\\
\leq\: & m_b + \sum_{i=1}^{h}|\beta_i|_b \\
\leq\: & n_r &\text{by~\eqref{prop-gentype2-R2.eq},}\\
\leq\: &n_r+\sum_{i=1}^{t'+t''+1}\left|\rho_{i}\right|_r \\
\leq\: &n_r'' & \text{by~\eqref{ineq.nrpp.4.11.case.4}.}
\end{align*}
Therefore, \eqref{prop-gentype2-R2.eq} is satisfied by $m_b''$, $n_r''$ and $(\beta_{s'+1},\ldots,\beta_{h-t''-1})$. Finally, we have
\begin{align*}
\: & \mu((\beta_{s'+1},\ldots,\beta_{h-t''-1}),m_b'',m_r'') \\
=\: &m_b''+m_r'' + \sum_{i=s'+1}^{h-t''-1}(|\beta_i|-1) - 1\\
\leq\: &m_b+m_r-t-t'-t''+\sum_{i=1}^{s'}(|\beta_i|-1) -1 + \sum_{i=s'+1}^{h-t''-1}(|\beta_i|-1) - 1 & \text{by~\eqref{ineq.summbmrpp.4.11.case.4},}\\
\leq\: &m_b+m_r+\sum_{i=1}^{h-t''-1}(|\beta_i|-1) - 2\\
<\: &m_b+m_r+\sum_{i=1}^{h}(|\beta_i| -1) - 1 \\
=\: & \mu((\beta_1,\ldots,\beta_h),m_b,m_r).
\end{align*}
\end{enumerate}
\end{enumerate}
\end{proof}

The following theorem is a direct consequence of Propositions~\ref{prop-gentype2-B}, \ref{prop-gentype2-R1} and~\ref{prop-gentype2-R2}.

\begin{theorem}\label{thm-gentype2}
Let $\mathbb{B}=(k_e,k_r,k_b,1,h)$ be of generalized type II board, $\mathscr{B}=\left(m_b, m_r\right)$, $\mathscr{R}=\left(n_b, n_r\right)$, and assume $\mathscr{B}$ is the active player.
Then, $\mathscr{B}$ has a winning strategy if and only if 
$$ m_b>0 \qquad \textrm{and} \qquad m_b+\sum_{i=1}^{h} |\beta_{i}|_b>n_r$$
when $\mathscr{B}$'s turn starts. Whenever there is a winning strategy for $\mathscr{B}$ (respectively for $\mathscr{R}$), then $\strat$ is such a strategy.
\end{theorem}

\section{Final Theorem}
\label{sec:FinalTheorem}

In this section, we consider general game states, that is, we do not make any assumption about the number of long piles on a board $\mathbb{B}$. If $\mathbb{B}$ contains $h\geq 1$ long $b$-piles, we denote the $h$ long $b$-piles by $\beta_{1}, \beta_{2}, \ldots, \beta_{h}$. If $\mathbb{B}$ contains $\ell\geq 1$ long $r$-piles, we denote the $\ell$ long $r$-piles by $\rho_{1}, \rho_{2}, \ldots, \rho_{\ell}$. The structure of this section is the same as the one for Subsection~\ref{subsec:gentypeII}. We first describe the states where there is a winning strategy for $\mathscr{B}$ on some board $\mathbb{B}$ (refer to Proposition~\ref{prop-gen-1}). Then we describe the states where there is a winning strategy for $\mathscr{R}$ on some board $\mathbb{B}$ (refer to Propositions~\ref{prop-gen-2} and~\ref{prop-gen-3}).
\begin{prop}
\label{prop-gen-1}
Let $\mathbb{B}=(k_e,k_r,k_b,\ell,h)$, $\mathscr{B}=\left(m_b, m_r\right)$, $\mathscr{R}=\left(n_b, n_r\right)$, and assume that $\mathscr{B}$ is the active player. If 
\begin{align}
\label{prop-gen-1.eq}
m_b>0 \qquad \textrm{and} \qquad \left(n_r=0 \quad \textrm{or} \quad m_b+\sum_{i=1}^{h}|\beta_i|_b > n_r+\sum_{i=1}^{\ell}|\rho_i|_r-\max_{1\leq i\leq \ell}\big\{|\rho_i|_r\big\}\right)
\end{align}
when $\mathscr{B}$'s turn starts, then $\mathscr{B}$ wins by applying $\strat$ at every round.
\end{prop}

\begin{proof}
Suppose $\mathbb{B}=(k_e,k_r,k_b,\ell,h)$ with $\ell,h\geq0$. Consider a game state where $\mathscr{B}=\left(m_b, m_r\right)$, $\mathscr{R}=\left(n_b, n_r\right)$, with $m_b>0$ and, either $n_r=0$ or $m_b+\sum_{i=1}^{h}|\beta_i|_b > n_r+\sum_{i=1}^{\ell}|\rho_i|_r-\max_{1\leq i\leq \ell}\big\{|\rho_i|_r\big\}$. Since $m_b>0$, then $\mathscr{B}$ has at least one $b$ chip. Player $\mathscr{B}$ applies Strategy $\strat$: they capture all $b$-piles and discard all their prisoners, as per steps 1 and 2 of $\strat$. The state of the board $\mathbb{B}$ is now the following: 
$$\mathbb{B} = (k_e+k_b+h,k_r,0,\ell,0).$$ 
We now have a board $\mathbb{B}$ of generalized type I, with active player $\mathscr{B}=(m_b',0)$ where $m_b' \geq m_b+\sum_{i=1}^{h}|\beta_i|_b > 0$. Moreover, $\mathscr{R}=(n_b',n_r')$ where $n_b' \leq n_b$ and $n_r' = n_r$.

If $n_r=0$, then by Theorem~\ref{thm-gentype1}, $\mathscr{B}$ wins by applying $\strat$ at every round. Otherwise, we have
$$m_b' \geq m_b+\sum_{i=1}^{h}|\beta_i|_b > n_r+\sum_{i=1}^{\ell}|\rho_i|_r-\max_{1\leq i\leq \ell}\big\{|\rho_i|_r\big\} = n_r'+\sum_{i=1}^{\ell}|\rho_i|_r-\max_{1\leq i\leq \ell}\big\{|\rho_i|_r\big\} $$
by~\eqref{prop-gen-1.eq}. Therefore, by Theorem~\ref{thm-gentype1}, $\mathscr{B}$ wins by applying $\strat$ at every round.
\end{proof}

\begin{prop}
\label{prop-gen-2}
Let $\mathbb{B}=(k_e,k_r,k_b,\ell,h)$, $\mathscr{B}=\left(m_b, m_r\right)$, $\mathscr{R}=\left(n_b, n_r\right)$, and assume that $\mathscr{B}$ is the active player. If $m_b=0$ when $\mathscr{B}$'s turn starts, then $\mathscr{R}$ wins by applying $\strat$ at every round.
\end{prop}

\begin{proof}
Suppose $\mathbb{B}=(k_e,k_r,k_b,\ell,h)$ with $\ell,h\geq 0$. Consider a game state where $\mathscr{B}=(m_b,m_r)$ with $m_b=0$ and $\mathscr{R}=(n_b,n_r)$.

If $m_r=0$, then $\mathscr{B}$ is out of chips, they are eliminated and $\mathscr{R}$ wins. Otherwise, $\mathscr{B}$ starts their round by making a sequence of $t+t'+t''\geq 0$ moves, that consist of placing (refer to Theorem~\ref{thm:SameActive}):
$$
\left\{
	\begin{array}{l}
		\text{$t$ $r$ chips on $t$ empty piles,}\\
		\text{$t'$ $r$ chips on $t'$ $b$-singletons,}\\
		\text{and $t''$ $r$ chips on $t''$ long $b$-piles.}
	\end{array}
\right.
$$
The state of the board $\mathbb{B}$ is now the following:
$$\mathbb{B}=(k_e-t,k_r+t,k_b-t',\ell+t'+t'',h-t''). $$

We have $\mathscr{B}=(0,m_r')$ where $m_r'\leq m_r-t-t'-t''$.
If $m_r'=0$, then $\mathscr{B}$ is out of chips, they are eliminated and $\mathscr{R}$ wins. Otherwise, $\mathscr{B}$ must end their round by placing an $r$ chip on an $r$-pile (refer to Theorem \ref{thm:DifferentActive}). This results in a capture for $\mathscr{R}$, from which $\mathscr{R}$ gains at least one $r$ chip. The move goes to $\mathscr{R}$, who captures all $r$-piles and discards all their prisoners, as per steps 1 and 2 of $\strat$. Now, the only long piles on the board are the $h-t''$ long $b$-piles. We consider two cases: either (1) $0\leq h-t''\leq 1$ or (2) $h-t''\geq 1$.
\begin{enumerate}
\item Suppose $0\leq h-t''\leq 1$. If $h - t'' = 0$, then, as per Step 3 of $\strat$, $\mathscr{R}$ places an $r$ chip on a $b$-singleton if there is one, or on an empty pile otherwise. If $h - t'' = 1$, then $\mathscr{R}$ places an $r$ chip on the only long $b$-pile on the board. In all cases, the board contains at most one long pile, which, if it exists, is a long $r$-pile. As such, $\mathbb{B}$ is of type I and it is $\mathscr{B}$’s turn to play. Since $m_b=0$, then by Theorem~\ref{thm-type1}, $\mathscr{R}$ wins by applying $\strat$ at every round.

\item If $h-t''>1$, then, as per Step 3 of $\strat$, $\mathscr{R}$ places an $r$ chip on the $b$-pile of greatest length. As such, $\mathbb{B}=(k_e+k_r+\ell,k_b,1,h-t''-1)$ is of generalized type II. Since $m_b=0$, then by Theorem~\ref{thm-gentype2}, $\mathscr{R}$ wins by applying $\strat$ at every round.
\end{enumerate}
\end{proof}

\begin{prop}
\label{prop-gen-3}
Let $\mathbb{B}=(k_e,k_r,k_b,\ell,h)$, $\mathscr{B}=\left(m_b, m_r\right)$, $\mathscr{R}=\left(n_b, n_r\right)$, and assume that $\mathscr{B}$ is the active player. If
\begin{equation}
\label{eq-general}
m_b > 0, \quad n_r > 0\quad \text{ and } \quad
m_b+\sum_{i=1}^{h}\left|\beta_{i}\right|_b \leq n_r+\sum_{i=1}^{\ell}\left|\rho_{i}\right|_r-\max_{1\leq i\leq \ell}\big\{\left|\rho_{i}\right|_r\big\}
\end{equation}
when $\mathscr{B}$'s turn starts, then $\mathscr{R}$ wins by applying $\strat$ at every round.
\end{prop}

\begin{proof}
Suppose $\mathbb{B}=(k_e,k_r,k_b,\ell,h)$ with $\ell,h\geq 0$. If $h=0$, then $\mathbb{B}$ is of generalized type I and~\eqref{eq-general} becomes
$$
m_b > 0, \quad n_r > 0\quad \text{ and } \quad
m_b \leq n_r+\sum_{i=1}^{\ell}\left|\rho_{i}\right|_r-\max_{1\leq i\leq \ell}\big\{\left|\rho_{i}\right|_r\big\}.
$$
Then, by Theorem~\ref{thm-gentype1}, $\mathscr{R}$ wins by applying $\strat$ at every round.

Now suppose $h \geq 1$ and $\mathscr{B}$ starts their round by making a sequence of $t+t'+t''+s+s'\geq 0$ moves, that consist of placing (refer to Theorem~\ref{thm:SameActive}):
$$
\left\{\begin{array}{l}
\text{$t$ $r$ chips on $t$ empty piles,} \\
\text{$t'$ $r$ chips on $t'$ $b$-singletons,} \\
\text{$t''$ $r$ chips on $t''$ long $b$-piles,} \\
\text{$s$ $b$ chips on $s$ $b$-singletons,} \\
\text{and $s'$ $b$ chips on $s'$ long $b$-piles.}
\end{array}\right.
$$
The state of the board $\mathbb{B}$ is now the following:
$$\mathbb{B}=(k_e-t+s+s',k_r+t,k_b-t'-s,\ell+t'+t'',h-t''-s'). $$
Let us denote the $s'$ long $b$-piles that were captured by $\beta_1,\ldots,\beta_{s'}$ and the remaining long $b$-piles on the board by $\beta_{s'+1},\ldots, \beta_{h-t''}$. The $t'+t''$ long $r$-piles that were created consist of $t'$ $(b,r)$-piles and $t''$ long $r$-piles $\langle\beta_{h-t''+1},(r)\rangle,\langle\beta_{h-t''+2},(r)\rangle,\ldots,\langle\beta_{h},(r)\rangle$. We rename the $t'+t''$ long $r$-piles by $\rho_{\ell+1},\rho_{\ell+2},\ldots,\rho_{\ell+t'+t''}$.

For each pile captured by $\mathscr{B}$, one chip must be discarded (by the Capture Rule). This chip can either be a guard or a prisoner.
Therefore, we have $\mathscr{B}=(m_b',m_r')$, where 
\begin{align*} 
m_b' &\leq m_b+\sum_{i=1}^{s'}|\beta_i|_b,\\
m_r'&\leq m_r-t-t'-t''+\sum_{i=1}^{s'}|\beta_i|_r,\\
m_b'+m_r' &\leq m_b+m_r-t-t'-t''+\sum_{i=1}^{s'}(|\beta_i|-1). 
\end{align*}
Player $\mathscr{B}$ can end their round in four different ways (refer to Theorem \ref{thm:DifferentActive}): (1) placing a $b$ chip on an empty pile, (2) placing a $b$ chip on an $r$-singleton, (3) placing a $b$ chip on a long $r$-pile, or (4) placing an $r$ chip on an $r$-pile. Observe that in all cases, \eqref{eq-general} implies that $n_r\geq 1$, i.e., $\mathscr{R}$ has at least one $r$ chip.
\begin{enumerate}
\item Suppose $\mathscr{B}$ places a $b$ chip on an empty pile. Then, a $b$-singleton is created and no long pile is created. 
The move goes to $\mathscr{R}$, who captures all $r$-piles and discards all their prisoners, as per steps 1 and 2 of $\strat$. Therefore, there are $h-t''-s'$ long piles, which are all long $b$-piles.
Then, as per Step 3 of $\strat$, $\mathscr{R}$ places an $r$ chip on the $b$-pile of greatest length. Thus, it is $\mathscr{B}$'s turn to play, and we have $\mathscr{B}=(m_b'',m_r'')$ and $\mathscr{R}=(0,n_r'')$, where
\begin{align}
\label{ineq.mbpp.5.3}
m_b'' &\leq m_b+\sum_{i=1}^{s'}|\beta_i|_b-1,\\
\label{ineq.mrpp.5.3}
m_r'' &\leq m_r-t-t'-t''+\sum_{i=1}^{s'}|\beta_i|_r,\\
\label{ineq.summbmrpp.5.3}
m_b''+m_r'' &\leq m_b+m_r-t-t'-t''+\sum_{i=1}^{s'}(|\beta_i|-1)-1,\\
\label{ineq.nrpp.5.3}
n_r'' &\geq n_r+\sum_{i=1}^{\ell+t'+t''}|\rho_i|_r -1.
\end{align}
The state of the board $\mathbb{B}$ depends on the value of $h-t''-s'$: either (a) $0\leq h-t''-s' \leq 1$ or (b) $h-t''-s' > 1$.
\begin{enumerate}
\item Suppose $0\leq h-t''-s'\leq 1$. If $h-t''-s' = 0$, then $\mathscr{R}$ ended their round by placing an $r$ chip on a $b$-singleton. If $h-t''-s' = 1$, then $\mathscr{R}$ ended their round by placing an $r$ chip on the only long $b$-pile that was on the board. In both cases, the board now contains only one long pile, which is a long $r$-pile. As such, $\mathbb{B}$ is of type I and it is $\mathscr{B}$'s turn to play. Observe that $m_b''<n_r''$. Indeed, we have
\begin{align*}
m_b'' \leq\: & m_b+\sum_{i=1}^{s'}|\beta_i|_b-1 &\text{by~\eqref{ineq.mbpp.5.3},}\\
\leq\: & m_b+\sum_{i=1}^{h-t''}\left|\beta_{i}\right|_b-1 &\text{since $0\leq h-t''-s'$,}\\
\leq\: & m_b+\sum_{i=1}^{h}\left|\beta_{i}\right|_b-1\\
\leq\: & n_r+\sum_{i=1}^{\ell}\left|\rho_{i}\right|_r-\max_{1\leq i\leq \ell}\big\{\left|\rho_{i}\right|_r\big\}-1 &\text{by~\eqref{eq-general},}\\
<\: & n_r+\sum_{i=1}^{\ell}|\rho_i|_r -1\\
\leq\: & n_r+\sum_{i=1}^{\ell+t'+t''}|\rho_i|_r -1\\
\leq\: & n_r'' & \text{by~\eqref{ineq.nrpp.5.3}.}
\end{align*}
Then, by Theorem~\ref{thm-type1}, $\mathscr{R}$ wins by applying $\strat$ at every round.

\item Suppose $h-t''-s' > 1$. Then $\mathscr{R}$ ended their round by placing an $r$ chip on the $b$-pile of greatest length, say $\beta_{h-t''}$ without loss of generality. Thus, the board contains one long $r$-pile $\langle\beta_{h-t''},(r)\rangle$ and $h-t''-s'-1 > 0$ long $b$-piles. As such, $\mathbb{B}$ is of generalized type II and it is $\mathscr{B}$'s turn to play. If $m_b'' = 0$, then by Theorem~\ref{thm-gentype2}, $\mathscr{R}$ wins by applying $\strat$ at every round. Otherwise, $m_b'' > 0$ and we have
\begin{align*}
m_b'' + \sum_{i=s'+1}^{h-t''-1} |\beta_i|_b \leq\: &  m_b+\sum_{i=1}^{s'}|\beta_i|_b-1 + \sum_{i=s'+1}^{h-t''-1} |\beta_i|_b & \text{by~\eqref{ineq.mbpp.5.3},}\\
=\: &  m_b+\sum_{i=1}^{h-t''-1}|\beta_i|_b-1 \\
<\: &  m_b+\sum_{i=1}^{h}|\beta_i|_b-1 \\
\leq\: & n_r+\sum_{i=1}^{\ell}\left|\rho_{i}\right|_r-\max_{1\leq i\leq \ell}\big\{\left|\rho_{i}\right|_r\big\}-1 &\text{by~\eqref{eq-general},}\\
<\: & n_r+\sum_{i=1}^{\ell}|\rho_i|_r -1\\
\leq\: & n_r+\sum_{i=1}^{\ell+t'+t''}|\rho_i|_r -1\\
\leq\: & n_r'' & \text{by~\eqref{ineq.nrpp.5.3}.}
\end{align*}
Then, by Theorem~\ref{thm-gentype2}, $\mathscr{R}$ wins by applying $\strat$ at every round.
\end{enumerate}

\item Suppose $\mathscr{B}$ places a $b$ chip on an $r$-singleton. Then, one long $b$-pile is created, namely $(r,b)$.
The move goes to $\mathscr{R}$, who captures all $r$-piles and discards all prisoners, as per steps 1 and 2 of $\strat$.
Therefore, there are $h - t'' - s'+1$ long piles, which are all long $b$-piles. Then, as per Step 3 of $\strat$, $\mathscr{R}$ places an $r$ chip on the $b$-pile of greatest length. Thus, it is $\mathscr{B}$'s turn to play, and we have
$\mathscr{B}=(m_b'',m_r'')$ and $\mathscr{R}=(0,n_r'')$, where~\eqref{ineq.mbpp.5.3},
\eqref{ineq.mrpp.5.3}, \eqref{ineq.summbmrpp.5.3} and~\eqref{ineq.nrpp.5.3} are satisfied.
The state of the board $\mathbb{B}$ depends on the value of $h-t''-s'+1$: either (a) $h-t''-s'+1 = 1$ or (b) $h-t''-s'+1 > 1$.
\begin{enumerate}
\item Suppose $h-t''-s'+1=1$. Then $\mathscr{R}$ ended their round by placing an $r$ chip on the only long $b$-pile that was on the board. As such, the board now contains only one long pile, which is a long $r$-pile. Thus, $\mathbb{B}$ is of type I and it is $\mathscr{B}$'s turn to play). Observe that $m_b''<n_r''$. The derivation for this inequality is the same as in Case 1(a). Thus, by Theorem~\ref{thm-type1}, $\mathscr{R}$ wins by applying $\strat$ at every round.

\item Suppose $h-t''-s'+1>1$. Then $\mathscr{R}$ ended their round by placing an $r$ chip on the $b$-pile of greatest length, say $\beta_{h-t''}$ without loss of generality. As such, the board now contains one long $r$-pile $\langle\beta_{h-t''},(r)\rangle$ and $h-t''-s' > 0$ long $b$-piles. Thus, $\mathbb{B}$ is of generalized type II and it is $\mathscr{B}$'s turn to play. If $m_b'' = 0$, then by Theorem~\ref{thm-gentype2}, $\mathscr{R}$ wins by applying $\strat$ at every round. Otherwise, $m_b'' > 0$ and we have
$$m_b'' + \sum_{i=s'+1}^{h-t''-1} |\beta_i|_b < n_r'' .$$
The derivation for this inequality is the same as in Case 1(b).
Then, by Theorem~\ref{thm-gentype2}, $\mathscr{R}$ wins by applying $\strat$ at every round.
\end{enumerate}

\item Suppose $\mathscr{B}$ places a $b$ chip on a long $r$-pile. Then, one long $b$-pile is created. Without loss of generality, let it be $\langle\rho_{\ell+t'+t''},(b)\rangle$.
The move goes to $\mathscr{R}$, who captures all $r$-piles and discards all their prisoners, as per steps 1 and 2 of $\strat$.
Therefore, there are $h-t''-s'+1$ long piles, which are all long $b$-piles.
Then, as per Step 3 of $\strat$, $\mathscr{R}$ places an $r$ chip on the $b$-pile of greatest length. Thus, it is $\mathscr{B}$'s turn to play, and we have $\mathscr{B}=(m_b'',m_r'')$ and $\mathscr{R}=(0,n_r'')$ where
\begin{align}
\label{ineq.mbpp.5.3.case3}
m_b'' &\leq m_b+\sum_{i=1}^{s'}|\beta_i|_b-1, \\
\label{ineq.mrpp.5.3.case3}
m_r'' &\leq m_r-t-t'-t''+\sum_{i=1}^{s'}|\beta_i|_r, \\
\label{ineq.summbmrpp.5.3.case3}
m_b''+m_r''&\leq m_b+m_r-t-t'-t''+\sum_{i=1}^{s'}(|\beta_i|-1)-1, \\
\label{ineq.nrpp.5.3.case3}
n_r'' &\geq n_r+\sum_{i=1}^{\ell+t'+t''-1}|\rho_i|_r -1.    
\end{align}
The state of the board $\mathbb{B}$ depends on the value of $h-t''-s'+1$: either (a) $h-t''-s'+1 = 1$ or (b) $h-t''-s'+1 > 1$.
\begin{enumerate}
\item Suppose $h-t''-s'+1=1$. Then $\mathscr{R}$ ended their round by placing an $r$ chip on the only long $b$-pile that was on the board. As such, the board now contains only one long pile, which is a long $r$-pile. Thus, $\mathbb{B}$ is of type I and it is $\mathscr{B}$'s turn to play. Observe that $m_b'' \leq n_r''$. Indeed, we have
\begin{align*}
m_b'' \leq\: & m_b+\sum_{i=1}^{s'}|\beta_i|_b-1 & \text{by~\eqref{ineq.mbpp.5.3.case3},}\\
\leq\: &m_b+\sum_{i=1}^{h-t''}\left|\beta_{i}\right|_b-1 &\text{since $0\leq h-t''-s'$,}\\
\leq\: & m_b+\sum_{i=1}^{h}\left|\beta_{i}\right|_b-1\\
\leq\: &n_r+\sum_{i=1}^{\ell}\left|\rho_{i}\right|_r-\max_{1\leq i\leq \ell}\big\{\left|\rho_{i}\right|_r\big\}-1 &\text{by~\eqref{eq-general},}\\
\leq\: &n_r+\sum_{i=1}^{\ell-1}\left|\rho_{i}\right|_r-1 &\text{since $|\rho_{\ell}|_r\leq \max_{1\leq i\leq \ell}\big\{\left|\rho_{i}\right|_r\big\}$,}\\
\leq\: &n_r+\sum_{i=1}^{\ell+t'+t''-1}|\rho_i|_r -1\\
\leq\: &n_r'' & \text{by~\eqref{ineq.nrpp.5.3.case3}.}
\end{align*}
Then, by Theorem~\ref{thm-type1}, $\mathscr{R}$ wins by applying $\strat$ at every round.

\item Suppose $h-t''-s'+1 > 1$. Then $\mathscr{R}$ ended their round by placing an $r$ chip on the $b$-pile $\pi_{\max}$ of greatest length, denoted by $\pi_{\max}$. Observe that 
$$\pi_{\max}\in\{\langle\rho_{\ell+t'+t''},(b)\rangle\}\cup\{\beta_i \mid s'+1\leq i \leq h-t'' \}.$$
Thus, the board contains one long $r$-pile $\langle\pi_{\max},(r)\rangle$ and $h-t''-s' > 0$ long $b$-piles. As such, $\mathbb{B}$ is of generalized type II and it is $\mathscr{B}$'s turn to play. If $m_b'' = 0$, then by Proposition~\ref{prop-gentype2-R1}, $\mathscr{R}$ wins by applying $\strat$ at every round. Otherwise, $m_b'' > 0$ and 
\begin{align*}
\: & m_b'' + \sum_{i=s'+1}^{h-t''}|\beta_i|_b + |\langle\rho_{\ell+t'+t''},(b)\rangle|_b - |\pi_{\max}|_b & \text{by~\eqref{ineq.mbpp.5.3.case3},}\\
\leq\: & m_b+\sum_{i=1}^{s'}|\beta_i|_b -1 + \sum_{i=s'+1}^{h-t''}|\beta_i|_b + |\langle\rho_{\ell+t'+t''},(b)\rangle|_b  - |\pi_{\max}|_b\\
=\: & m_b+\sum_{i=1}^{h-t''}|\beta_i|_b + |\langle\rho_{\ell+t'+t''},(b)\rangle|_b  - |\pi_{\max}|_b -1\\
\leq\: & m_b + \sum_{i=1}^{h}|\beta_i|_b -1 &\text{since $|\langle\rho_{\ell+t'+t''},(b)\rangle|_b  \leq |\pi_{\max}|_b$,}\\
\leq\: &n_r+\sum_{i=1}^{\ell}\left|\rho_{i}\right|_r-\max_{1\leq i\leq \ell}\big\{\left|\rho_{i}\right|_r\big\}-1 &\text{by~\eqref{eq-general},}\\
\leq\: &n_r+\sum_{i=1}^{\ell-1}\left|\rho_{i}\right|_r-1 &\text{since $|\rho_{\ell}|_r\leq \max_{1\leq i\leq \ell}\big\{\left|\rho_{i}\right|_r\big\}$,}\\
\leq\: &n_r+\sum_{i=1}^{\ell+t'+t''-1}|\rho_i|_r -1\\
\leq\: &n_r'' & \text{by~\eqref{ineq.nrpp.5.3.case3}.}
\end{align*}
Then, by Theorem~\ref{thm-gentype2}, $\mathscr{R}$ wins by applying $\strat$ at every round.
\end{enumerate}

\item Suppose $\mathscr{B}$ places an $r$ chip on an $r$-pile. This results in a capture for $\mathscr{R}$, from which $\mathscr{R}$ gains at least one $r$ chip. The move goes to $\mathscr{R}$, who captures all $r$-piles and discards all their prisoners, as per steps 1 and 2 of $\strat$. Therefore, there are $h-t''-s'$ long piles, which are all long $b$-piles. Then, as per Step 3 of $\strat$, $\mathscr{R}$ places an $r$ chip on the $b$-pile of greatest length. Thus, it is $\mathscr{B}$'s turn to play, and we have $\mathscr{B}=(m_b'',m_r'')$ and $\mathscr{R}=(0,n_r'')$ where
\begin{align} 
\label{ineq.mbpp.5.3.case4}
m_b'' &\leq m_b+\sum_{i=1}^{s'}|\beta_i|_b,\\
\label{ineq.mrpp.5.3.case4}
m_r'' &\leq m_r-t-t'-t''+\sum_{i=1}^{s'}\left|\beta_{i}\right|_r-1,\\
\label{ineq.summbmrpp.5.3.case4}
m_b''+m_r'' &\leq m_b+m_r-t-t'-t''+\sum_{i=1}^{s'}(|\beta_i|-1)-1,\\
\label{ineq.nrpp.5.3.case4}
n_r' &\geq n_r+\sum_{i=1}^{\ell+t'+t''}|\rho_i|_r.
\end{align}
The state of the board $\mathbb{B}$ depends on the value of $h-t''-s'$: either (a) $0\leq h-t''-s' \leq 1$ or (b) $h-t''-s' > 1$.
\begin{enumerate}
\item Suppose $0\leq h-t''-s'\leq 1$. If $h-t''-s' = 0$, then $\mathscr{R}$ ended their round by placing an $r$ chip either on a $b$-singleton if $ $ or on an empty pile of $ $, as per Step 3 of $\strat$. If $h-t''-s' = 1$, then $\mathscr{R}$ ended their round by placing an $r$ chip on the only long $b$-pile on the board, as per Step 3 of $\strat$. In all cases, the board now contains at most one long pile, which is a long $r$-pile if it exists. As such, $\mathbb{B}$ is of type I and it is $\mathscr{B}$’s turn to play. Observe that $m_b''< n_r''$. Indeed, we have
\begin{align*}
m_b'' \leq\: & m_b+\sum_{i=1}^{s'}|\beta_i|_b & \text{by~\eqref{ineq.mbpp.5.3.case4},}\\
\leq\: &m_b+\sum_{i=1}^{h-t''}\left|\beta_{i}\right|_b &\text{since $0\leq h-t''-s'$,}\\
\leq\: & m_b+\sum_{i=1}^{h}\left|\beta_{i}\right|_b\\
\leq\: &n_r+\sum_{i=1}^{\ell}\left|\rho_{i}\right|_r-\max_{1\leq i\leq \ell}\big\{\left|\rho_{i}\right|_r\big\} &\text{by~\eqref{eq-general},}\\
<\: &n_r+\sum_{i=1}^{\ell}\left|\rho_{i}\right|_r \\
\leq\: &n_r+\sum_{i=1}^{\ell+t'+t''}|\rho_i|_r \\
\leq\: &n_r'' & \text{by~\eqref{ineq.nrpp.5.3.case4}.}
\end{align*}
Then, by Theorem~\ref{thm-type1}, $\mathscr{R}$ wins by applying $\strat$ at every round.

\item Suppose $h-t''-s'>1$. Then $\mathscr{R}$ ended their round by placing an $r$ chip on the $b$-pile of greatest length, say $\beta_{h-t''}$ without loss of generality. Thus, the board contains one long $r$-pile $\langle\beta_{h-t''},(r)\rangle$ and $h-t''-s'-1 > 0$ long $b$-piles. As such, $\mathbb{B}$ is of generalized type II and it is $\mathscr{B}$'s turn to play. If $m_b'' = 0$, then by Theorem~\ref{thm-gentype2}, $\mathscr{R}$ wins by applying $\strat$ at every round. Otherwise, $m_b'' > 0$ and we have
\begin{align*}
m_b'' + \sum_{i=s'+1}^{h-t''-1} |\beta_i|_b \leq\: &  m_b+\sum_{i=1}^{s'}|\beta_i|_b + \sum_{i=s'+1}^{h-t''-1} |\beta_i|_b & \text{by~\eqref{ineq.mbpp.5.3.case4},}\\
=\: &  m_b+\sum_{i=1}^{h-t''-1}|\beta_i|_b \\
<\: &  m_b+\sum_{i=1}^{h}|\beta_i|_b \\
\leq\: &n_r+\sum_{i=1}^{\ell}\left|\rho_{i}\right|_r-\max_{1\leq i\leq \ell}\big\{\left|\rho_{i}\right|_r\big\} &\text{by~\eqref{eq-general},}\\
<\: &n_r+\sum_{i=1}^{\ell}|\rho_i|_r \\
\leq\: &n_r+\sum_{i=1}^{\ell+t'+t''}|\rho_i|_r \\
\leq\: &n_r'' & \text{by~\eqref{ineq.nrpp.5.3.case4}.}
\end{align*}
\end{enumerate}
Then, by Theorem~\ref{thm-gentype2}, $\mathscr{R}$ wins by applying $\strat$ at every round.

\end{enumerate}
\end{proof}

The following (final) theorem is a direct consequence of Propositions~\ref{prop-gen-1}, \ref{prop-gen-2} and~\ref{prop-gen-3}.
\begin{theorem}
Let $\mathbb{B}=(k_e,k_r,k_b,\ell,h)$, $\mathscr{B}=\left(m_b, m_r\right)$, $\mathscr{R}=\left(n_b, n_r\right)$, and assume that $\mathscr{B}$ is the active player. Then, $\mathscr{B}$ has a winning strategy if and only if 
$$m_b>0 \qquad \textrm{and} \qquad \left( n_r=0 \quad\textrm{ or }\quad m_b+\sum_{i=1}^{h}\left|\beta_{i}\right|_b>n_r+\sum_{i=1}^{\ell}\left|\rho_{i}\right|_r-\max_{1\leq i\leq \ell}\big\{\left|\rho_{i}\right|_r\big\}\right)$$
when $\mathscr{B}$'s turn starts. Whenever there is a winning strategy for $\mathscr{B}$ (respectively for $\mathscr{R}$), then $\strat$ is such a strategy.
\end{theorem}

\section{Conclusion}

The mathematicians who developed So Long Sucker had a vision of creating a board game meant to highlight concepts of game theory such as individual and group behavior, power dynamics, and the psychology behind decision-making. Common themes seen among players include deception and betrayal. There are anecdotes of people playing SLS that leave you questioning the strength of your own relationships. Poudstone~\cite{Poundstone1992} comments on how spouses would leave in separate cabs after a night out playing SLS.

So Long Sucker stands out from other board games due to its simple presentation combined with crafty rules. When all four players are still playing in a game of SLS, it remains impossible to determine with certainty which player will win. However, with the results of this paper, an observer of a two-player two-color game can immediately determine which player has a winning strategy. 

Next, it is of interest to characterize first-player Blue's winning strategies in a game with two players, Blue and Red; and three colors, blue, red, and green. Once there are three colors in play, the chips of the eliminated player affect the game differently than the rest of the chips. For instance, let us suppose that Green is the eliminated player. Then, a pile captured with a green chip gets discarded; and, the next player chosen would be whoever gave Green their last move (see Next Player Rule). Thus, adding an eliminated player's chips to the two-player game significantly affects Blue's strategies. We conjecture that this situation can be studied in a manner similar to the two-color situation. From there, we would continue to study winning strategies in order to establish which player has a winning strategy when there are only two players left in the game and all four colors. By studying exhaustively the two-player case, this would allow the game to end as soon as a second player is eliminated. From then on, one would question studying the three-player case. However, given the versatility of the rules, studying the three-player case might require different tools than those seen in this paper.

Another point of interest is describing the Nash Equilibrium of a given SLS game. When there are only two players left in the game, the game pay-off comes down to win-lose. Thus, the losing player has no incentive to play any particular way, as they will lose regardless of how they play, and their pay-off does not change. If the chips had value, then the losing player would aim to lose as few chips as possible in order to minimize their loss.  

Game theory has long had applications in economics. Thus, if the game is approached from an economical perspective, it would be useful to use players' behavior in order to study their risk aversion. It is of interest to understand the players' attempt to reduce uncertainty and their perception of cost and benefit when accepting or rejecting a deal \cite{Samuelson16}.

So Long Sucker has a lot of potential in helping to understand social dynamics between players. For instance, it is useful for studying people's personality traits, such as their eagerness to lead, their bargaining skills, and their principles. As noted in Section~\ref{section.introduction}, this is something Hofstede and Murff touched on \cite{HofstedeTiptonMurff11}, as well as Blachard et al.~\cite{Hofstra14}.

\paragraph{Acknowledgments.}
The authors of this paper thank Dr. Pieter Hofstra for introducing them to So Long Sucker. Dr. Hofstra was a category theorist who dabbled in game theory during his free time. He was extremely interested in using So Long Sucker to study people's behavior \cite{Hofstra14}. Dr. Hofstra sadly passed away before the results came to fruition, and the authors would like to dedicate this paper to his name. 

\bibliographystyle{abbrv}
\bibliography{biblio}

\end{document}